\documentclass{amsart}

\usepackage{amssymb,amsmath,amsthm,amscd,graphicx,pxfonts,tikz-cd,enumitem,thmtools}
\usepackage[all,cmtip]{xy}

\theoremstyle{plain}
\newtheorem*{thm*}{Theorem}
\newtheorem*{lm*}{Lemma}
\newtheorem{thm}{Theorem}
\newtheorem{lm}[thm]{Lemma}
\newtheorem{cor}[thm]{Corollary}
\newtheorem{prop}[thm]{Proposition}

\newtheorem{que}[thm]{Question}

\theoremstyle{definition}

\newtheorem{de}[thm]{Definition}
\declaretheorem[sibling=thm,name=Example,qed={$\clubsuit$}]{ex}

\newtheorem{re}[thm]{Remark}
\newtheorem{no}[thm]{Notation}

\DeclareMathOperator{\RR}{{\mathbb R}}

\DeclareMathOperator{\ZZ}{{\mathbb Z}}

\DeclareMathOperator{\FF}{{\mathbb F}}

\let\AA\relax\DeclareMathOperator{\AA}{{\mathbb A}}

\DeclareMathOperator{\Set}{\mathbf{Set}}
\let\Top\relax\DeclareMathOperator{\Top}{\mathbf{Top}}
\DeclareMathOperator{\ModR}{\mathbf{Mod}_R}
\DeclareMathOperator{\ModB}{\mathbf{Mod}_B}

\DeclareMathOperator{\AlgR}{\mathbf{Alg}_R}
\DeclareMathOperator{\AlgB}{\mathbf{Alg}_B}
\DeclareMathOperator{\DomR}{\mathbf{Dom}_R}

\DeclareMathOperator{\fgModR}{\mathbf{fgMod}_R}

\DeclareMathOperator{\fgfMod}{\mathbf{fgfMod}}
\DeclareMathOperator{\fgfModR}{\mathbf{fgfMod}_R}
\DeclareMathOperator{\fgfModB}{\mathbf{fgfMod}_B}
\DeclareMathOperator{\PFR}{\mathbf{PF}_R}
\DeclareMathOperator{\PFA}{\mathbf{PF}_A}

\newcommand{\fp}{\mathfrak{p}}
\newcommand{\fq}{\mathfrak{q}}
\newcommand{\kbar}{\overline{K}}
\newcommand{\kpbar}{\overline{K_{\fp}}}

\DeclareMathOperator{\cha}{char}
\DeclareMathOperator{\id}{id}
\DeclareMathOperator{\im}{im}
\DeclareMathOperator{\red}{red}
\DeclareMathOperator{\rad}{rad}
\DeclareMathOperator{\rk}{rk}

\DeclareMathOperator{\End}{End}
\DeclareMathOperator{\Frac}{Frac}
\DeclareMathOperator{\GL}{GL}
\DeclareMathOperator{\Gr}{Gr}
\DeclareMathOperator{\Hom}{Hom}
\DeclareMathOperator{\Sh}{Sh}
\DeclareMathOperator{\Spec}{Spec}
\DeclareMathOperator{\Sp}{Sp}
\DeclareMathOperator{\Forget}{Forget}

\DeclareMathOperator{\bd}{\mathbf{d}}
\DeclareMathOperator{\cI}{\mathcal{I}}
\DeclareMathOperator{\cV}{\mathcal{V}}
\let\phi\relax\DeclareMathOperator{\phi}{\varphi}

\author{Arthur Bik}
\address{University of Bern, Switzerland, and MPI for Mathematics in the Sciences, Germany}
\email{arthur.bik@mis.mpg.de}

\author{Alessandro Danelon}
\address{Eindhoven University of Technology, The Netherlands}
\email{a.danelon@tue.nl}

\author{Jan Draisma}
\address{Universit\"at Bern, Switzerland, and Eindhoven University of Technology, The Netherlands}
\email{jan.draisma@unibe.ch}

\begin{document}
\title[Topological Noetherianity of polynomial functors II]{Topological Noetherianity of polynomial functors II: base rings with Noetherian spectrum}

\begin{abstract}
In a previous paper, the third author proved that finite-degree polynomial
functors over infinite fields are topologically Noetherian. In this paper,
we prove that the same holds for polynomial functors from free $R$-modules
to finitely generated $R$-modules, for any commutative ring $R$ whose
spectrum is Noetherian. As Erman-Sam-Snowden pointed out, when applying
this with $R=\ZZ$ to direct sums of symmetric powers, one of their proofs
of a conjecture by Stillman becomes characteristic-independent.

Our paper advertises and further develops the beautiful but not so
well-known machinery of polynomial laws. In particular, to any finitely
generated $R$-module~$M$ we associate a topological space, which we show
is Noetherian when $\Spec(R)$ is; this is the degree-zero case of our
result on polynomial functors.
\end{abstract}

\maketitle

\section{Introduction and main theorem}

\subsection{Summary}
A polynomial functor over an infinite field $K$ is a functor
$P$ from the category of finite-dimensional $K$-vector spaces to itself
such that for any two finite-dimensional vector spaces $V,W$ the map
$P_{V,W}\colon\Hom(V,W) \to \Hom(P(V),P(W))$ is a polynomial map. In many
respects, polynomial functors behave like univariate polynomials: they can
be added (direct sums), multiplied (tensor products), and composed; they
are direct sums of unique homogeneous polynomial functors of degrees $0,1,2,\ldots$;
and---for the theory that we are about to develop quite importantly---they
can be shifted by a constant: if $P$ is a polynomial functor and $U$
a constant vector space, then the functor $\Sh_U(P)$ that assigns to
$V$ the vector space $P(U \oplus V)$ and to $\phi \in \Hom_K(V,W)$ the
linear map $P(\id_U \oplus\,\phi)$ is a polynomial functor. Furthermore,
if $P$ has finite degree, which we will always require, then---much
like a univariate polynomial and its shift by a constant---$\Sh_U(P)$
has the same degree, and the top-degree homogeneous components of $P$
and $\Sh_U(P)$ are canonically isomorphic. 

From a different perspective, polynomial functors are the ambient spaces
of ``$\GL_\infty$-equivariant algebraic geometry'', a research area
which has seen much activity over the last years. A closed subset of $P$
is a rule $X$ that assigns to a vector space $V$ a Zariski-closed subset
$X(V)$ of $P(V)$ in such a manner that for each $\phi \in \Hom(U,V)$,
the linear map $P_{U,V}(\phi)$ maps $X(U)$ into $X(V)$. In earlier
work~\cite{draisma}, the third author showed that if $P$ has finite
degree, then it is Noetherian in the sense that any descending chain
of closed subsets $P \supseteq X_1 \supseteq X_2 \supseteq\ldots$
eventually stabilises. This was used in work by Erman-Sam-Snowden
\cite{erman-sam-snowden,erman-sam-snowden2,erman-sam-snowden3} and by
Draisma-Laso\'n-Leykin \cite{draisma-lason-leykin} in new proofs of the
conjecture by Stillman that the projective dimension of a homogeneous
ideal that is generated by a fixed number of forms of a fixed degree is
uniformly bounded independently of the number of variables \cite[Problem
3.14]{Peeva09}.  In this context, Erman-Sam-Snowden asked whether the
Noetherianity of polynomial functors also holds over $\ZZ$; this would
show that their proof of Stillman's conjecture yields bounds that are
independent of the characteristic, just like another 
proof by Erman-Sam-Snowden \cite{erman-sam-snowden}
and the original proof by Ananyan-Hochster \cite{ananyan-hochster}.

In this paper, we settle Erman-Sam-Snowden's question in the
affirmative. Indeed, rather than working over $\ZZ$, we will work over a
ring $R$ whose spectrum is Noetherian---this turns out to be
precisely the setting where topological Noetherianity also holds
for polynomial functors. 

So let $R$ be a ring (commutative with $1$). In
Section~\ref{sec:Polynomial laws} we will review the notion of {\em
polynomial laws} from an $R$-module $M$ to an $R$-module $N$.  In the
special case where $N=R$, these polynomial laws form a graded ring $R[M]$
(see \S\ref{ssec:RM}), where the notation is chosen to resemble that
for the coordinate ring of an affine variety. This ring will be used in
Section~\ref{sec:The topological space AM} to define a topological space
$\AA_M$, in such a manner that any polynomial law $\phi\colon M \to N$
yields a continuous map, also denoted $\phi$, from $\AA_M \to \AA_N$.
To be precise, $\AA_M$ is a topological space over the category $\DomR$
of $R$-domains with $R$-algebra monomorphisms. Here a topological space
over a category $\mathbf{C}$ is not a single set, but a functor from
$\mathbf{C}$ equipped with the notions of elements and (closed) subsets,
and we let all definitions related to usual topological spaces stated in
terms of their elements and (closed) subsets carry over to this
setting; see Definition~\ref{de:TopSpace} for details. 

If $M$ is freely generated by $n$ elements, then $R[M]$ is the polynomial
ring $R[x_1,\ldots,x_n]$ and 
the poset of closed sets in $\AA_M$ is the same as that in the spectrum of
$R[M]$. In general, however, we do not completely understand the relation
between $\AA_M$ and the spectrum of $R[M]$ (see Remark~\ref{re:SpecRM}),
and we work with the former rather than the latter. The following result
is a topological version of Hilbert's basis theorem in this setting.

\begin{prop} \label{prop:TopHilbert}
If $R$ has a Noetherian spectrum and $M$ is a finitely generated
$R$-module, then the topological space $\AA_M$ over $\DomR$ is Noetherian.
\end{prop}

Interestingly, it is not true that if $R$ is Noetherian and $M$ is
finitely generated, then $R[M]$ is Noetherian (see
Example~\ref{ex:NonNoetherian}),
so ``topologically Noetherian'' is the most natural setting here.
A special case of the theorem (taking $M$ free of rank $1$) is that if $R$
has a Noetherian spectrum, then so does the polynomial ring $R[x]$. This
special case, a topological version of Hilbert's basis theorem, is easy
and well-known; e.g., it also follows from \cite[Theorem 1.1]{erman-sam-snowden4}
with a trivial group $G$.

Following \cite{roby}, in Section~\ref{sec:Polynomial functors over a ring} we will recall the notion of
polynomial functors from the category $\fgfModR$ of finitely generated
free $R$-modules to the category $\ModR$ of $R$-modules. These polynomial
functors form an Abelian category. The subcategory of polynomial functors
from $\fgfModR$ to the category $\fgModR$ of finitely generated,
but not necessarily free, 
$R$-modules is not an Abelian subcategory when $R$ is not Noetherian,
but it is closed under taking quotients, and this will suffice for
our purposes.

Given a polynomial functor $P\colon\fgfModR\to\fgModR$,
a closed subset of $\AA_P$ 
is a rule~$X$ that assigns to each finitely generated free $R$-module $U$
a closed subset $X(U)$ of $\AA_{P(U)}$ such that the
continuous map corresponding to the polynomial law 
\[ 
\Hom(U,V) \times P(U) \to P(V),\quad (\phi,p) \mapsto
P_{U,V}(\phi)(p) 
\]
maps the pre-image of $X(U)$ under the projection on $P(U)$ in $\AA_{\Hom(U,V) \times
P(U)}$ into $X(V)$ (see \S\ref{ssec:ClosedSubsets} for
details). 
If $Y$ is a second such rule, then we say that $X$ is a subset of $Y$ if
$X(U)$ is a subset of $Y(U)$ for each $U \in \fgfModR$. Our main result,
then, is the following.

\begin{thm} \label{thm:Main}
Let $R$ be a commutative ring whose spectrum is a Noetherian topological
space and let $P$ be a finite-degree polynomial functor $\fgfModR
\to \fgModR$. Then every descending chain $X_1 \supseteq X_2 \supseteq
\ldots$ of closed subsets of $\AA_P$ stabilises: for all sufficiently large~$n$
 we have $X_n=X_{n+1}$.
\end{thm}

Proposition~\ref{prop:TopHilbert} is the special case of Theorem~\ref{thm:Main}
where the polynomial functor has degree $0$, i.e., sends each $U$
to a fixed module $M$ and each morphism to the identity
$\id_M$. Proposition~\ref{prop:TopHilbert} will be proved first, as a base case
in an inductive proof of Theorem~\ref{thm:Main}.

\subsection{Structure of the paper} \label{ssec:Structure}
In Section~\ref{sec:preliminary}, we establish
and recall certain basic results. In Section~\ref{sec:Polynomial laws} we define polynomial laws and the coordinate ring of a module over a ring. Section~\ref{sec:The topological space AM} is devoted to the topological space $\AA_M$. Here we prove Proposition~\ref{prop:TopHilbert}, the first fundamental
fact needed for our inductive proof of Theorem~\ref{thm:Main}. 

Then, in Section~\ref{sec:Polynomial functors over a ring} we recall the definition a polynomial functor $P$ over a ring and several of its properties. Among these is the Friedlander-Suslin
lemma that yields equivalences of Abelian categories between polynomial
functors $\fgfModR\to\fgModR$ of degree $\leq d$
and finitely generated modules for the
non-commutative $R$-algebra $R[\End(U)]_{\leq d}^*$ (called the Schur
algebra) for any $U \in \fgfModR$ of rank~$\geq d$. We also prove the second fundamental
fact needed for Theorem~\ref{thm:Main}: if $R$ is a
domain and $P$ a polynomial functor from $\fgfModR$ to $\ModR$ such
that $\Frac(R) \otimes P$ is irreducible, then $\Frac(R/\fp) \otimes
P$ is irreducible for all primes $\fp$ in some open dense subset of
$\Spec(R)$. This is an incarnation of the philosophy in representation
theory that irreducibility is a generic condition.

Finally, in Section~\ref{sec:Proof} we prove Theorem~\ref{thm:Main}. The
global proof strategy is as follows: we show that the induction steps
in \cite{draisma}, where Theorem~\ref{thm:Main} is proved when $R$ is
an infinite field, can be made global in the sense that they hold for
$\Frac(R/\fp)$ for all $\fp$ in some open dense subset of $\Spec(R)$;
and then we use Noetherian induction on $\Spec(R)$ to deal with the
remaining primes $\fp$. The details of this approach are a quite subtle
and beautiful.

The big picture is depicted in the following diagram:
\begin{center}
\includegraphics[width=\textwidth]{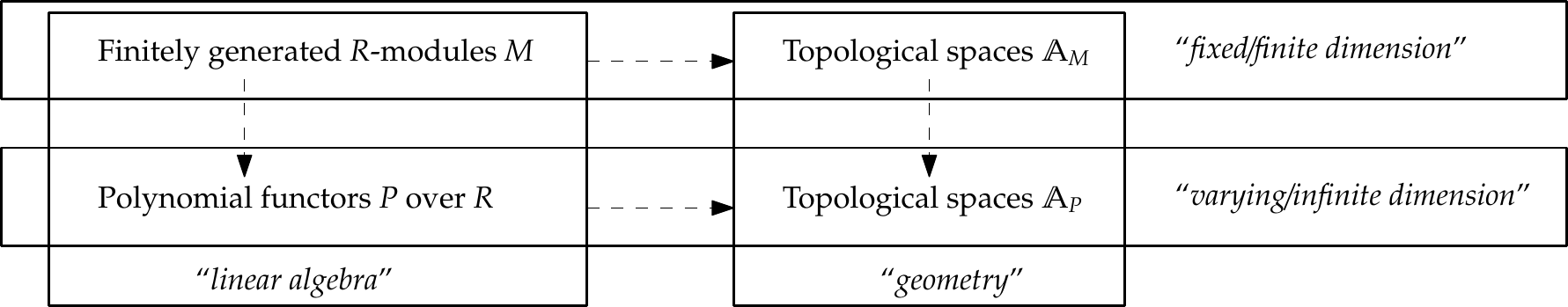}
\end{center}

\noindent
Building on the notion of finitely generated $R$-modules, on the left
we pass to polynomial functors over $R$. Here many results carry over,
such as the fact that the rank is a semicontinuous function on $\Spec(R)$;
see Proposition~\ref{prop:DimensionFunction}.  We regard this as ``linear
algebra in varying dimensions''. In the other direction, we construct
the topological space $\AA_M$ and enter the realm of algebraic geometry;
the closed subsets generalise affine algebraic varieties. Finally, both
constructs come together in the construction of the topological space
associated to a polynomial functor $P$. Here we use both results from the
``linear algebra'' of polynomial functors, such as Friedlander-Suslin's
lemma, and results about the topological spaces $\AA_M$, to prove that
$\AA_P$ is Noetherian. Furthermore, we establish the fundamental
result that the dimension function of a closed subset of $\AA_P$
depends on primes in $\Spec(R)$ in a constructible manner; see
Proposition~\ref{prop:Last}.

\subsection{A class of applications} \label{ssec:Applications}

Our original motivation for this paper is the following: let $P,Q$
be (finite-degree) polynomial functors from the category of finitely
generated free $\ZZ$-modules to itself and let $\alpha\colon Q \to P$ be
a polynomial transformation; see Definition~\ref{de:PolyTrans}. Define
the closed subset $X$ of $\AA_P$ as the closure of the image of
$\alpha$. Specifically, for a natural number $n$, the pull-back
along $\alpha_{\ZZ^n}$ defines a ring homomorphism $\ZZ[P(\ZZ^n)]
\to \ZZ[Q(\ZZ^n)]$, and $X(\ZZ^n)$ is the closed subset of $\Spec
\ZZ[P(\ZZ^n)]$ defined by the kernel of that ring homomorphism.
Theorem~\ref{thm:Main} implies the following. 

\begin{cor} \label{cor:OverZ}
There exists a uniform bound $d$ such that for all $n \in \ZZ_{\geq 0}$
and all fields $K$, $X(K^n) \subseteq K \otimes P(\ZZ^n)$ is defined by
polynomials of degree $\leq d$.
\end{cor}

This corollary has many applications; here is one. If $V$ is a
finite-dimensional vector space over a field $K$ and $T \in V \otimes V
\otimes V$ is a tensor, then $T$ is said to have {\em slice rank} 
$\leq r$ if $T$ can be written as the sum of $r$ terms of the form $\sigma(v
\otimes A)$, where $v \in V$ and $A \in V \otimes V$, and $\sigma$ is
a cyclic permutation of $1,2,3$ permuting the tensor factors. If $K$
is algebraically closed, then being of slice rank $\leq r$ is a
Zariski-closed condition \cite{sawin-tao}.

\begin{cor} \label{cor:SliceRank}
Fix a natural number $r$. There exists a uniform bound $d$ such that for
all algebraically closed fields $K$ and for all $n \in \ZZ_{\geq 0}$,
the variety of slice-rank-$\leq r$ tensors in $K^n \otimes K^n
\otimes K^n$ is defined by polynomials of degree $\leq d$.
\end{cor}

The same holds when the number of tensor factors is increased to
any fixed number, possibly at the expense of increasing $d$, and
similar results hold for the set of cubic forms of bounded q-rank
\cite{derksen-eggermont-snowden} or for the closure of the set of
degree-$e$ forms of bounded strength in the sense of \cite{ananyan-hochster}.
We stress, however, that ``defined by'' is intended in a purely
set-theoretic sense. We do not know whether the vanishing ideals of
these varieties are generated in bounded degree, even if the field $K$
were fixed beforehand.

\begin{proof}[Proof of Corollary~\ref{cor:SliceRank}.]
Consider the polynomial functor $P$ that sends a free
$\ZZ$-module $\ZZ^n$
to $\ZZ^n \otimes \ZZ^n \otimes \ZZ^n$, and the polynomial
functor $Q$ that sends $\ZZ^n$
to $\ZZ^n \oplus (\ZZ^n \otimes \ZZ^n)$. For any $r$-tuple
$(\sigma_1,\ldots,\sigma_r)$
of cyclic permutations of $1,2,3$ we have a polynomial transformation
\[ Q^r \to P, ((v_1,A_1),\ldots,(v_r,A_r)) \mapsto \sum_{i=1}^r
\sigma_i(v_i \otimes A_i), \]
whose image closure is defined in uniformly bounded degree $e$ by
Corollary~\ref{cor:OverZ}. The variety of slice-rank-$\leq r$ tensors is
the union of these image closures over all $r$-tuples of cyclic permutations,
hence defined in degree at most $e \cdot 3^r$, independently of 
the algebraically closed field and independently of $n$. 
\end{proof}

\begin{re} \label{re:ZForms}
Over a field $K$ of characteristic zero, the irreducible polynomial
functors~$P$ are precisely the Schur functors, and any polynomial
functor is isomorphic to a direct sum of Schur functors. These always admit
a $\ZZ$-form, i.e., a polynomial functor $P_{\ZZ}$ over $\ZZ$ such that $K
\otimes P_{\ZZ} \cong P$, which moreover has the property that it maps free
$\ZZ$-modules to free $\ZZ$-modules \cite{akin-buchsbaum-weyman}. The $\ZZ$-form need not be
unique; e.g., the Schur functor over $K$ that maps $V$ to its
$d$-th symmetric power $S^d V$, comes both from the functor from free $\ZZ$-modules
to free $\ZZ$-modules that sends $U$ to $S^d U$ and from the functor
that sends $U$ to the sub-$\ZZ$-module of $U^{\otimes d}$ consisting of
symmetric tensors. These two functors are non-isomorphic $\ZZ$-forms.
In applications such as the above, where one looks for field-independent
bounds, it is important to choose the $\ZZ$-form that captures the
problem of interest.
\end{re}

\begin{ex}
Again over $R=\ZZ$, consider the polynomial transformation $\alpha\colon
(S^2)^4 \to S^4$ that maps a quadruple $(q_1,\ldots,q_4)$
of quadratic forms to $q_1^2 + \cdots + q_4^2$. Let $X$ be the image
closure as above. If $K$ is algebraically closed of characteristic zero,
then $X_K(K^4)$ is a hypersurface in $S^4(K)$ of degree 38475
\cite{Blekherman12},
so the degree bound from Corollary~\ref{cor:OverZ} must be at least
that large. On the other hand, if $K$ is algebraically closed of
characteristic $2$, then the image of $\alpha$ is just the linear space
spanned by all degree-four monomials that are squares, and hence only
linear equations are needed to cut out this image.
\end{ex}

\begin{re}
Over algebraically closed fields of positive characteristic, irreducible
polynomial functors are still parameterised by partitions, but
polynomial functors are no longer semisimple, and the $\ZZ$-forms
from Remark~\ref{re:ZForms} do not always remain irreducible; standard
references are \cite{carter-lusztig,green}. The typical example is that, in
characteristic $p$, the functor $S^p$ contains a subfunctor that maps $V$
to the linear space of $p$-th powers of elements of $V$.
\end{re}

\subsection{Further relations to the literature}

The polynomial functors that we study are often referred to as strict
polynomial functors in the literature; we will drop the adjective
``strict''. We do not know whether the polynomial functors over finite
fields studied in \cite{pirashvili} admit a similar theory.

Much literature on polynomial functors is primarily concerned with
representation theory, whereas our emphasis is on the geometry/commutative
algebra of closed subsets in such polynomial functors.

We will use work of Roby on polynomial laws \cite{roby} and work
of Touz\'e on polynomial functors \cite{touze}---but indeed only
more elementary parts of their work, such as the generalisation of
Friedlander-Suslin's \cite[Theorem 3.2]{friedlander-suslin} to general
base rings $R$; see \cite[Th\'eor\`eme 7.2]{touze}.

The paper \cite{erman-sam-snowden3} establishes finiteness results
for (cone-stable and weakly upper semi-continuous) ideal invariants in
polynomial rings over a fixed field. As Erman pointed out to us, at least
part of their results carry over to arbitrary base rings with Noetherian
spectrum. In particular, Erman-Sam-Snowden establish the Noetherianity of
a space $Y_{\bd}$ that parameterises homogeneous ideals generated in degrees
$\bd=(d_1,\ldots,d_r)$. While they work with certain limit
spaces, the ``functor analogue'' of their $Y_{\bd}$ in our setting would be a
functor from $\fgfModR$ to the category of functors from $\DomR$ to sets
that sends a finitely generated free $R$-module $U=R^n$ to the functor
that maps an $R$-domain $D$ to the set of $\GL_n(D)$-orbits of ideals
in $R[x_1,\ldots,x_n]$ generated by homogeneous polynomials of degrees
$d_1,\ldots,d_r$. Then $Y_{\bd}$ admits a surjective map from the space
$\AA_{S^{d_1} \oplus \cdots \oplus S^{d_r}}$---a functor from $\fgfModR$
to functors from $\DomR$ to topological spaces, and one can give $Y_{\bd}$
the quotient topology. Theorem~\ref{thm:Main} implies that $Y_{\bd}$
is then Noetherian, provided that $\Spec(R)$ is Noetherian.

Our work does not say much about Noetherianity of the coordinate rings
$R[\AA_P]$, let alone about Noetherianity of finitely generated modules
over them. Currently, these much stronger results are known only when
$R$ is a field of characteristic zero and $P$ is either a direct
sum of copies of $S^1$ \cite{sam-snowden1,sam-snowden2} or $P=S^2$
or $P=\bigwedge^2$ \cite{nagpal-sam-snowden} or $P=S^1 \oplus S^2$
or $P=S^1 \oplus \bigwedge^2$ \cite{sam-snowden3}.

Like Ananyan-Hochster's work \cite{ananyan-hochster}, recent work by
Kazhdan and Ziegler \cite{kazhdan-ziegler1,kazhdan-ziegler2} implies
that polynomials of high strength, and high-strength sequences of
polynomials, behave very much like generic polynomials or sequences.
Like Corollary~\ref{cor:OverZ}, their results are uniform in the
characteristic of the field.  But the route that Kazhdan and Ziegler take
is entirely different: first a theorem is proved over finite fields
by algebraic-combinatorial means, with uniform constants that do not
depend on the finite field, and then model theory is used to transfer
the result to arbitrary algebraically closed fields.

In \cite{bik-draisma-eggermont} it is shown that in any closed subset
of the polynomial functor $S^d$ defined over $\ZZ$, the strength of
polynomials over a ground field of characteristic $0$ or characteristic
$>d$ is uniformly bounded from above. While of a similar flavour as
Corollary~\ref{cor:OverZ}, that result---in which the restriction on
the characteristic cannot be removed---does not follow from our current
work. Far-reaching generalisations of \cite{bik-draisma-eggermont},
but only over fields of characteristic zero, are the topic of the 
forthcoming preprint \cite{bik-draisma-eggermont-snowden}. 

\section{Preliminaries}\label{sec:preliminary}

\subsection{Rings and algebras}
Throughout the paper, all rings are commutative and with $1$ and ring
homomorphisms are required to be unital. We fix a ring $R$, and
if $\fp$ is a prime ideal in $R$, then we write $K_\fp$ for the fraction field
of the domain $R/\fp$. If $R$ is a domain, then we write $K:=K_{(0)}$ for
the fraction field of $R$. 

An $R$-algebra is an (unless otherwise stated) commutative ring with a
homomorphism from $R$ into it; an $R$-algebra homomorphisms from an $A$
to $B$ is a ring homomorphism $A \to B$ such that composition of the
homomorphisms $R \to A \to B$ is the prescribed homomorphism $R \to
B$. Except where specified otherwise, tensor products are over $R$,
$\Hom(U,V)$ is the $R$-module of $R$-module homomorphisms from $U$
to~$V$, and $U^*=\Hom(U,R)$. We use the terms $R$-domain and $R$-field
for $R$-algebras that, as rings, are domains and fields, respectively.

\subsection{From finitely generated to free modules.\!} \label{ssec:Modules}
The following lemma, which we will later generalise to
polynomial functors, is well-known; we give a proof for
completeness.

\begin{lm} \label{lm:GenericFreeness}
Let $R$ be a domain, let $M$ be a finitely generated $R$-module, and
let $N$ be a submodule of $M$. Then there exists a nonzero $r \in R$
and elements $v_1,\ldots,v_n \in N$ such that $R[1/r] \otimes N$ is a
finitely generated free submodule of $R[1/r] \otimes M$ with basis $1
\otimes v_1, \ldots, 1 \otimes v_n$, and such that $R[1/r] \otimes M$
is the direct sum of $R[1/r] \otimes N$ and another free $R[1/r]$-module.
\end{lm}

Note that tensoring with $K$ yields that $n=\dim_{K}(K \otimes N)$.

\begin{proof}
The vector space $K\otimes N$ is contained in the finite-dimensional
vector space $K \otimes M$. Hence there exist $v_1,\ldots,v_n \in N$ such
that $1 \otimes v_1, \ldots, 1 \otimes v_n$ is a basis of $K \otimes N$,
and $v_{n+1}, \ldots, v_m \in M$ such that $1 \otimes v_{n+1}, \ldots,
1 \otimes v_m$ is a basis of a complement of $K \otimes N$ in $K \otimes
M$. We claim that both statements hold with $K$ replaced by $R[1/r]$
for some nonzero $r$.

To see this, extend $v_1,\ldots,v_m$ with $v_{m+1},\ldots,v_l$ to a
generating set of the $R$-module $M$. Then for each $j=m+1,\ldots,l$
we have, in $K \otimes M$,
\[ 
1 \otimes v_j=\sum_{i=1}^m c_{ij} \otimes v_i 
\]
for certain coefficients $c_{ij} \in K$. This identity means
that there exists a non-zero element $r\in R$ and suitable coefficients
$c_{ij}$'s in $R$ such that 
\[ 
1 \otimes v_j = \sum_{i=1}^m (c'_{ij}/r) \otimes v_i 
\]
holds in $R[1/r] \otimes M$. 
Hence $R[1/r] \otimes M$ is generated
by $1 \otimes v_1, \ldots, 1 \otimes v_m$, and these elements do not
have any nontrivial linear relation over $R[1/r]$ since
their images in $K \otimes M$ do not
satisfy any such relation over $K$. It follows that $R[1/r] \otimes M$
is free with basis $1 \otimes v_1,\ldots,1 \otimes v_m$.
Furthermore, $R[1/r] \otimes N$ contains the
$R[1/r]$-module spanned by $1 \otimes v_1,\ldots,1 \otimes
v_n$; and conversely, if $v \in R[1/r] \otimes M$ is an element of $R[1/r] \otimes N$, then it
cannot have a nonzero coefficient on any of the last $m-n$ basis elements,
because in $K \otimes M$ the image of $v$ is a linear
combination of the first $m$ basis elements and the basis
elements do not satisfy any linear relation there. Hence $R[1/r]
\otimes N \subseteq R[1/r] \otimes M$ is free with basis $1 \otimes v_1,\ldots,1
\otimes v_n$.
\end{proof}

\section{Polynomial laws and the coordinate ring of a module}
\label{sec:Polynomial laws}

\subsection{Polynomial laws}
We follow \cite[Chapter 1]{roby}. Let $M,N$ be $R$-modules. Denote by $\AlgR$ the category of $R$-algebras.

\begin{de}
A {\em polynomial law} $\phi\colon M \to N$ is a collection of
maps 
\[
(\phi_A:A \otimes M \to A \otimes N)_{A\in\AlgR}
\]
such that for every $R$-algebra homomorphism $\alpha\colon A \to B$ the following
diagram commutes:
\[ 
\xymatrix{ A \otimes M \ar[r]^{\phi_A} \ar[d]_{\alpha
\otimes \id_M} & A \otimes N
\ar[d]^{\alpha \otimes \id_N} \\
B \otimes M \ar[r]^{\phi_B} \ar[r] & B \otimes N.}
\]
\end{de}

\begin{ex}
Suppose that $M$ and $N$ are the free modules $R^2$ and $R$,
respectively, so that $A \otimes M$ and $A \otimes N$ are canonically
identified with $A^2$ and $A$. Then the collection $(\phi_A)_A$
defined by $\phi_A(x,y)=xy+y^2$ for $x,y \in A$ is a polynomial law $M
\to N$,
and indeed one that is {\em homogeneous} of degree $2$ in the sense of Definition
\ref{de:Homogeneous} below.
\end{ex}

More generally, the name polynomial law derives from the following fact.

\begin{lm} \label{lm:Injection}
Consider two $R$-modules $M$ and $N$. Suppose that $M$ is finitely generated and let $\{v_1,\ldots,v_n\}$ be a set of generators. Let $\phi\colon M\to N$ be a polynomial law. Then $\phi$ is completely determined by the element:
\[ 
\iota(\phi):=\phi_{R[x_1,\ldots,x_n]}(x_1 \otimes v_1 + \cdots + x_n \otimes
v_n)  \in R[x_1,\ldots,x_n] \otimes N. 
\]
This gives an injective map $\iota$ from the collection of polynomial laws from $M$ to $N$ to the module $R[x_1,\ldots,x_n] \otimes N$.
In the case where $M$ is free with basis $v_1,\ldots,v_n$,
this injective map is a bijection.
\end{lm}
\begin{proof}
Let $A$ be an $R$-algebra, let $a_1,\ldots,a_n\in A$ be elements and let $\alpha\colon R[x_1,\ldots,x_n]\to A$ be the $R$-algebra homomorphism sending $x_i\mapsto a_i$. Then the diagram associated to $\alpha$ shows that $\phi_A(a_1\otimes v_1+\cdots+a_n\otimes v_n)=(\alpha\otimes\id_N)\iota(\phi)$ and hence $\iota$ is injective. If $M$ is free with basis $v_1,\ldots,v_n$, then $\phi_A(a_1\otimes v_1+\cdots+a_n\otimes v_n)=\sum_jf_j(a_1,\ldots,a_n)\otimes w_j$ defines a polynomial law $\phi\colon M\to N$ for every $\sum_j f_j\otimes w_j\in R[x_1,\ldots,x_n]\otimes N$.
\end{proof}

\begin{ex} 
If $R$ is an infinite field, then a polynomial law $\phi$ from $M=R^n$
to $N=R^m$ is in fact uniquely determined by $\phi_R$, which is required to be
a polynomial map, i.e., a map all of whose $m$ coordinate functions are
polynomials in the $n$ coordinates on $M$. So then the set of polynomial
laws from $M$ to $N$ is precisely the set of polynomial maps from the
vector space $M$ to the vector space $N$.

For a general ring $R$, we denote by $\AA_R^n$ the affine scheme
$\Spec(R[x_1, \cdots, x_n])$. The set of polynomial laws from $R^n$ to $R^m$
is the set of morphisms $\AA_R^n \to \AA_R^m$ defined over $R$.
Of
course, such a morphism need not be determined by its map $\phi_R\colon R^n
\to R^m$, but it {\em is} determined by the maps $\phi_A\colon A^n \to A^m$
for all $R$-algebras $A$. This motivates the definition of
polynomial laws.
\end{ex}

\begin{de} \label{de:Homogeneous}
A polynomial law $\phi\colon M \to N$ is {\em homogeneous of degree $d$} 
if for each $R$-algebra $A$ and all $a \in A, m \in A \otimes M$, we have
$\phi_A(am)=a^d \phi_A(m)$. 
\end{de}

Writing $R[x_1,\ldots,x_n]_d$ for the set of homogeneous polynomials
of degree $d$, we see that the injection from Lemma~\ref{lm:Injection}
maps a homogeneous polynomial law $M \to N$ of degree $d$ to an element
of $R[x_1,\ldots,x_n]_d \otimes N$.

\begin{prop}\label{prop:composition}
Let $M_1,\ldots,M_d,N$ be $R$-modules and let $\phi\colon M_1\times\cdots\times M_d\to N$ be a multilinear map. Then $\phi$ extends to a homogeneous polynomial law of degree $d$ (also denoted $\phi$). After identifying $A\otimes(M_1\times\cdots\times M_d)\cong A\otimes M_1\times\cdots\times A\otimes M_d$, we have
$$
\phi_A\left(\sum_{i_1}a_{i_1}\otimes m_{i_1},\ldots,\sum_{i_d}a_{i_d}\otimes m_{i_d}\right)=\sum_{i_1,\ldots,i_d}a_{i_1}\cdots a_{i_d}\otimes\phi(m_{i_1},\ldots,m_{i_d})
$$
for all $R$-algebras $A$, $a_{i_1},\ldots,a_{i_d}\in A$ and $m_{i_1}\in M_1,\ldots,m_{i_d}\in M_d$.
\end{prop}
\begin{proof}
The maps $\phi_A$ are well-defined as the maps $A^d\times M_1\times\cdots\times M_d\to A \otimes N$ sending $(a_1,\ldots,a_d,m_1,\ldots,m_d)\mapsto a_1\cdots a_d\phi(m_1\cdots m_d)$ are multilinear. The collection $(\phi_A)_A$ is a homogeneous polynomial law of degree $d$ and $\phi_R=\phi$.
\end{proof}

\begin{re}\label{rem:composition}
Composition of $R$-module homomorphisms is a bilinear map. By the proposition, we can thus view this operation as a polynomial law.
\end{re}

A homogeneous polynomial law $\phi\colon M \to N$ of degree $0$ is the same
thing as an element of $N$ (namely, the element $\phi_R(0)$, which equals
$\phi_A(m)$ for any $R$-algebra $A$ and any element $m \in A \otimes M$);
we call these polynomial laws {\em constant}. A homogeneous polynomial law
$M \to N$ of degree $1$ is the extension of an $R$-module homomorphism
$M \to N$ as in the proposition above (namely, the map $\phi_R\colon M \to N$, which 
in this case is $R$-linear and uniquely determines
$\phi_A$ for all $A \in \AlgR$); we call these polynomial laws {\em linear}. 

The following proposition says that, in many ways, polynomial laws behave
like ordinary polynomial maps between vector spaces. For proofs we refer
to \cite{roby}.

\begin{prop} \label{prop:PolyLaws}
Let $\phi,\psi\colon M \to N$, $\gamma\colon N \to O$ be polynomial laws between
$R$-modules.
\begin{enumerate}
\item The collection $\phi+\psi:=(\phi_A+\psi_A)_A$ is a polynomial law $M
\to N$, homogeneous of degree $d$ if $\phi,\psi$ are. 

\item We have $\phi=\sum_{d=0}^\infty \phi_d$ for unique polynomial laws $\phi_d\colon M
\to N$ of degree $d$, where for each $R$-algebra $A$ and each $m \in A
\otimes M$ we have $\phi_{d,A}(m)=0$ for all but finitely many $d$'s ($\phi_d$
is called the {\em homogeneous component} of $\phi$ of
degree $d$); moreover, if $M$ is finitely generated, then
only finitely many of the $\phi_d$ are nonzero. 

\item The collection $\gamma \circ \phi:=(\gamma_A \circ \phi_A)_A$ is a
polynomial law $M \to O$, homogeneous of degree $d \cdot e$ if
$\phi,\psi$ are homogeneous of degrees $d,e$, respectively. 

\item If $N=R$, then $\phi \cdot \psi=(m \mapsto \phi_A(m)
\psi_A(m))_A$ is a polynomial law $M \to R$, homogeneous of
degree $d+e$ if $\phi,\psi$ are homogeneous of degrees
$d,e$, respectively.
\end{enumerate}
\end{prop}

\begin{prop}
Let $\phi\colon M\oplus M'\to N$ be a polynomial law between $R$-modules. Then $\phi$ has a unique decomposition $\phi=\sum_{i,j=0}^\infty\phi_{(i,j)}$ such that $\phi_{(i,j)}\colon M\oplus M'\to N$ is a bihomogeneous polynomial law of degree $(i,j)$, i.e., after identifying $A\otimes(M\oplus M')\cong A\otimes M\oplus A\otimes M'$, we have $\phi_{(i,j),A}(am,bm')=a^ib^j\phi_{(i,j),A}(m,m')$ for all $R$-algebras $A$, $a,b\in A$, $m\in A\otimes M$ and $m'\in A\otimes M'$. Moreover, if $\phi$ is homogeneous of degree $d$, then $\phi_{(i,j)}=0$ for all $i+j\neq d$.
\end{prop}
\begin{proof}
Suppose that such a decomposition exists and let $A$ be an $R$-algebra. Then we have
\[
\phi_{A[s,t]}(sm,tm')=\sum_{i,j}\phi_{(i,j),A[s,t]}(sm,tm')=\sum_{i,j}s^it^j\phi_{(i,j),A}(m,m')\in\bigoplus_{i,j=0}^\infty s^it^jA\otimes N
\]
for all $m\in A\otimes M$ and $m'\in A\otimes M'$. This shows that the $\phi_{(i,j)}$ are unique. If $\phi$ is homogeneous of degree $d$, setting $s=t$, we see that $\phi=\sum_{i+j=d}\phi_{(i,j)}$ and hence $\phi_{(i,j)}=0$ for $i+j\neq d$. What remains to show the existence of the decomposition. In fact, defining $\phi_{(i,j),A}(m,m')$ to be the coefficient of $s^it^j$ in $\phi_{A[s,t]}(sm,tm')$, it is easy to show that the $\phi_{(i,j)}$ are bihomogeneous polynomial laws of degree $(i,j)$ adding up to $\phi$. 
\end{proof}

The class of $R$-modules, in addition to its structure of Abelian
category with $R$-module homomorphisms as morphisms, has the structure of
a (non-Abelian) category with polynomial laws as morphisms. Both structures
will be important to us, but we reserve the notation $\ModR$ for the
category in which the morphisms are $R$-module homomorphisms (i.e.,
homogeneous polynomial laws of degree $1$).

\begin{de}[Base change] \label{de:BaseChangePolyLaw}
If $B$ is an $R$-algebra, then the tensor product functor $\ModR
\to \ModB$, which sends {\em linear} polynomial laws over $R$ to
linear polynomial laws over $B$, can be extended to a functor from the
category of $R$-modules with polynomial laws over $R$ to the category
of $B$-modules with polynomial laws over~$B$: on objects, the functor
is just $M \mapsto B \otimes M$, and a polynomial
law $(\phi_A)_{A \in \AlgR}\colon M \to N$ is mapped to
$(\phi_A)_{A \in \AlgB}$ where, for a $B$-algebra $A$, the
map $\phi_A$ is interpreted as a map $A \otimes_B (B
\otimes_R M) \cong A \otimes_R M \to A \otimes_R N \cong A \otimes_B
(B \otimes_R N)$.
\end{de}

\subsection{The coordinate ring of a module}
\label{ssec:RM}

Let $M$ be a finitely generated $R$-module. 

\begin{de} \label{de:RM}
We write $R[M]$ for the set of polynomial laws $M\to R$ and $R[M]_d\subseteq R[M]$ for the subset of homogeneous polynomial laws of degree $d$. The addition and multiplication from Proposition~\ref{prop:PolyLaws}, the grading from Definition~\ref{de:Homogeneous} and the identification $R[M]_0=R$ give $R[M]=\bigoplus_{d=0}^\infty R[M]_d$ the structure of a $\ZZ_{\geq0}$-graded commutative $R$-algebra. We call this $R$-algebra the {\em
coordinate ring} of $M$.
\end{de}

\begin{re}
In \cite[Chapitre III]{roby}, various algebras associated to an $R$-module~$M$ 
are introduced, but they are different from our $R$-algebra
$R[M]$. One important difference is that for us, the elements of $M$
play the role of geometric objects, whereas there, the algebras consist
of elements in divided or symmetric powers of~$M$.
\end{re}

As usual with coordinate rings, the association $M \mapsto R[M]$ is a
contravariant functor from the category of $R$-modules with polynomial
laws to the category of $R$-algebras: a polynomial law $\phi\colon M \to N$
has a pull-back map $\phi^*\colon R[N] \to R[M]$ sending $f \mapsto f \circ \phi$. If $\phi$
is linear, then $\phi^*$ is a graded homomorphism.

If $M$ is generated by $v_1,\ldots,v_n$, then the injection $\iota\colon R[M]\to R[x_1,\ldots,x_n]$ of
Lemma~\ref{lm:Injection} is a graded ring homomorphism. The following lemma says
precisely which subalgebra its image is. 

\begin{lm} \label{lm:GradedSubring}
Let $\psi\colon N \to M$ be a surjective $R$-module homomorphism. Then the map $\psi^*$ is a graded
isomorphism from $R[M]$ to the graded $R$-subalgebra of
$R[N]$ whose degree-$d$ part equals
\[ 
\{f \in R[N]_d \mid \forall u \in \ker(\psi): f \circ t_u = f \} 
\]
where $t_u:N \to N$ (called {\em translation by $u$}) is the affine-linear polynomial law $v \mapsto v+u$.
\end{lm}
\begin{proof}
Let $g\in R[M]_d$ and write $f=\psi^*(g)=g\circ \psi$. To see that $\psi^*$ is injective, note that $f_A=g_A\circ(\id_A\otimes\,\psi)$ for all $R$-algebras $A$. So if $f_A=0$, then $g_A=0$ as $\id_A\otimes\,\psi$ is surjective. To see that the image is contained in the subalgebra, it is enough to note that $\psi_A=\id_A\otimes\,\psi$ and $t_{u,A}(m)=m+1\otimes u$ and so $\psi\circ t_u=\psi$ as polynomial laws. Now, let $f \in R[N]_d$ be a polynomial law such that $f\circ t_u = f$ for all $u\in \ker(\psi)$. It remains to show that $f=g\circ\psi$ for some $g\in R[M]_d$. As $\id_A\otimes\,\psi$ is surjective, we set $g_A(m):=f_A(n)$ for any $n\in A\otimes N$ mapping to $m$. To do this, we need to show that $f_A(n)=f_A(n')$ whenever $n-n'\in\ker(\id_A\otimes\,\psi)$. Since the functor $A \otimes-$ from $R$-modules to $A$-modules is right-exact, we have $\ker(\id_A\otimes\,\psi)=A \otimes \ker(\psi)$. Take $h=f\circ((n,n')\mapsto n+n')$. Then we see that 
\[
h_A(n,1\otimes u)=f_A(n+1\otimes u)=(f\circ t_u)_A(n)=f_A(n)=h_A(n,0)
\]
for all $R$-algebras $A$, $n\in A\otimes N$ and $u\in\ker(\psi)$. It follows that $h_{(i,j),A}(n,1\otimes u)=0$ whenever $j>0$. And, we have $h_{(d,0),A}(n,n')=f_A(n)$. So
\[
\begin{array}{rll}
f_A(n+a\otimes u)&=h_A(n,a\otimes u)&\\
&=h_{(d,0),A}(n,a\otimes u)+\sum_{i=1}^dh_{(d-i,i),A}(b,a\otimes u)&\\
&=f_A(n)+\sum_{i=1}^da^ih_{(d-i,i),A}(b,1\otimes u)\\
&=f_A(n)
\end{array}
\]
for all $n\in A\otimes N$, $a\in A$ and $u\in\ker(\psi)$. So if $n-n'\in\ker(\id_A\otimes\,\psi)$, then $f_A(n)=f_A(n')$. This shows $g_A$ is well-defined. It is straightforward to check that $g=(g_A)_A$ is a homogeneous polynomial law of degree $d$. 
\end{proof}

\begin{ex}
When $R$ is an infinite field and both $M$ and $N$ are finite-dimensional
vector spaces over $R$, $R[M]$ is just the subring of $R[N]$
consisting of all polynomials that are constant on fibres of
the projection $N \to M$.  
\end{ex}

The following example shows that, even when $R$ is Noetherian and $M$
is finitely generated, $R[M]$ need not be Noetherian.

\begin{ex} \label{ex:NonNoetherian}
Let $R:=K[t]/(t^2)$ where $K$ is a field of characteristic zero, and
let $M:=K[t]/(t)$. Then $M=R/(t)$ is an $R$-module generated by a single
element $v:=1+(t)$ and $R[M]$ is the subring of $R[x]$ spanned by all
homogeneous polynomials $f=c x^d$ such that $f(x+at)=f(x)$ for all
$a \in K$. Now $c(x+at)^d=cx^d + cdatx^{d-1}$ and hence we need that $c \in
(t)$ whenever $d \geq 1$. Hence $R[M]$ is the vector space over $K$ spanned by
$1,t,tx,tx^2,\ldots$ with the multiplication $(t^ix)(t^ix)=0$. Observe
that $R[M]$ is not Noetherian, since the ideal $\mathrm{span}\{t,tx,tx^2,\ldots
\}$ is not finitely generated. On the other hand, the
quotient $R[M]^{\red}$ of $R[M]$ by its ideal of nilpotent elements is $K$.
\end{ex}

However, we will see later that if $\Spec(R)$ is Noetherian and $M$ is
finitely generated, then a certain topological space $\AA_M$ defined
using $R[M]$ is also Noetherian. In
Example~\ref{ex:NonNoetherian}, this is a
consequence of the fact that $\Spec(R[M])=\Spec(K)$ is Noetherian. See
also Remark~\ref{re:SpecRM}.

\begin{ex} \label{ex:CharTwo}
Now consider a field $K$ of characteristic $2$ and set
$R:=K[t]/(t^2)$. The same computation as above shows that $cx^i$
with {\em odd} $i$ can only be in $R[M] \subseteq R[x]$ if $c$ is in
$(t)$. But for {\em even} $i$, $cx^i$ is in $R[M]$ regardless of $c
\in R$. Hence $R[M]$ is the $K$-vector space with basis
\[ 
1,t,tx,x^2,tx^2,tx^3,x^4,tx^4,\ldots
 \]
and $R[M]^{\red} \cong K[x^2]$ as a graded algebra. 
\end{ex}

If $B$ is an $R$-algebra, then the base change functor from
Definition~\ref{de:BaseChangePolyLaw} sends polynomial laws $M \to
R$ to polynomial laws $B \otimes M \to B$. This yields an $R$-algebra
homomorphism $R[M] \to B[B\otimes M]$ and hence a $B$-algebra homomorphism
$B \otimes R[M] \to B[B \otimes M]$. The following example shows that
this needs not be an isomorphism.

\begin{ex}
Let $R=\ZZ$ and $M=\ZZ/2\ZZ$, generated by a single element
$v=1+2\ZZ$. Then by Lemma~\ref{lm:GradedSubring}, $R[M]$ is the
subring of $R[x]$ spanned by all homogeneous univariate polynomials $f$
such that $f(x+2a)=f(x)$ for all $a \in \ZZ$. Only the constant
polynomials have that property, so $R[M]=R$. Now take the $\ZZ$-algebra
$B=\ZZ/2\ZZ=:\FF_2$, which is a field, and $B \otimes M$ is the
one-dimensional vector space over that field, so $B[B \otimes M] \cong
\FF_2[x]$. 
\end{ex}

However, when $B$ is a localisation of a domain $R$, then the
map {\em is} an isomorphism:

\begin{prop}  \label{prop:Localisation}
Suppose that $R$ is a domain. Let $M$ be a finitely generated $R$-module and let $S$ be
a multiplicative subset of $R$ not containing $0$. Set $R':=S^{-1} R$. Then
\[ 
R'\otimes R[M] \cong S^{-1} R[M] \cong R'[R' \otimes M]\cong R'[S^{-1}M]. 
\]
\end{prop}

\begin{proof}
The first and last isomorphisms are standard. For the middle isomorphism,
we choose generators $m_1,\ldots,m_n$ of $M$ and embed $R[M]$ as a
graded $R$-subalgebra~$A$ of $R[x_1,\ldots,x_n]$. Since localisation
is exact, $S^{-1}R[M]$ is then isomorphic to the $R'$-algebra $S^{-1}
A \subseteq R'[x_1,\ldots,x_n]$. On the other hand, using the generators
$1 \otimes m_1,\ldots,1 \otimes m_n$, the $R'$-algebra $R'[R' \otimes M]$ also embeds as a
graded $R'$-subalgebra~$B$ of $R'[x_1,\ldots,x_n]$. The canonical map $R'
\otimes R[M] \to R'[R' \otimes M]$ translates into an inclusion $S^{-1}
A \subseteq B$, so it remains to show that $B \subseteq S^{-1} A$. For
this, let $O$ be the kernel of the $R$-module homomorphism $R^n \to M$
given by the generators $m_1,\ldots,m_n$. Again since localisation is
exact, $S^{-1}O \cong R' \otimes O$ is the kernel of the corresponding
$R'$-module homomorphism $(R')^n \to R' \otimes M$. Let $f \in B$ and
let $s \in S$ be such that $g:=sf \in R[x_1,\ldots,x_n]$. Then, since
$f \in B$, one has that $f \circ t_u = f$ for all $u \in S^{-1}O \subseteq (R')^n$,
by Lemma~\ref{lm:GradedSubring} applied to the $R'$-module $R' \otimes
M$. In particular, the multiplication by $s$ gives $g \circ t_u = g$ over $R'$ 
for all $u \in O \subseteq R^n$. Since $R$ is a domain, the same holds over $R$ 
and hence $g \in A$, again by Lemma~\ref{lm:GradedSubring} but now applied
to the $R$-module $M$. Hence $f=s^{-1} g \in S^{-1} A$, as desired.
\end{proof}

Like in ordinary algebraic geometry, the coordinate ring of
a direct sum is the tensor product of the coordinate rings. 

\begin{prop} \label{prop:Products}
Let $M,N$ be finitely generated $R$-modules. 
Then 
\[
R[M \oplus N] \cong R[M]\otimes R[N].
\]
\end{prop}
\begin{proof}
Elements of $R[M]$ and $R[N]$ induce elements of $R[M\oplus N]$ via composition with the projections $M\oplus N\to M$ and $M\oplus N\to N$, respectively. The product of such induced polynomial laws $M\oplus N\to R$ gives a bilinear map $R[M]\times R[N]\to R[M\oplus N]$. This induces an $R$-linear map $R[M]\otimes R[N]\to R[M\oplus N]$, which is in fact a homomorphism of $R$-algebras. Denote by $R[M\oplus N]_{(d,e)}$ the $R$-submodule of $R[M\oplus N]$ consisting of all bihomogeneous polynomial laws of degree $(d,e)$. It suffices to show that $R[M\oplus N]_{(d,e)}\cong R[M]_d\otimes R[N]_e$. To see this, first suppose that $M,N$ are free. In this case, we get $R[x_1,\ldots,x_n,y_1,\ldots,y_m]_{(d,e)}\cong R[x_1,\ldots,x_n]_d\otimes R[y_1,\ldots,y_m]_e$ when $x_i,y_j$ have degrees $(1,0),(0,1)$, respectively. In general, let $\phi\colon M'\to M$ and $\psi\colon N'\to N$ be surjective $R$-linear maps from finitely generated free $R$-modules. Then we see that
\[
\{f \in R[M'\oplus N']_{(d,e)} \mid \forall u_1 \in \ker(\phi)\forall u_1\in\ker(\psi) : f \circ t_{(u_1,u_2)} = f \} \cong
\]
\[
\{f \in R[M']_d \mid \forall u_1 \in \ker(\phi): f \circ t_{u_1} = f \}\otimes
\{g \in R[N']_e \mid \forall u_1\in\ker(\psi) : g \circ t_{u_2} = g\} 
\]
and hence $R[M\oplus N]_{(d,e)}\cong R[M]_d\otimes R[N]_e$.
\end{proof}

Example~\ref{ex:NonNoetherian} shows that the coordinate ring of
a module is quite a subtle notion. However, we will see that in
the proof of our Theorem~\ref{thm:Main}, by a localisation we
can always pass to a case where the module $M$ is free. In that
case, by Lemma~\ref{lm:GradedSubring}, $R[M]$ is just a
polynomial ring over $R$. 

\section{The topological space $\AA_M$}
\label{sec:The topological space AM}

\subsection{The space $\AA_M$}

We now construct the topological space $\AA_M$ for $M$ a finitely generated
$R$-module. To be precise,
$\AA_M$ is a topological space over the category $\DomR$ of $R$-domains with 
$R$-algebra monomorphisms, in the sense of the following definition. 

\begin{de} \label{de:TopSpace}
Let $F\colon\mathbf{C}\to\mathbf{D}$ be a functor and suppose that the
objects of $\mathbf{D}$ are sets and the morphisms are maps (i.e, we
have a forgetful functor $\Forget\colon\mathbf{D}\to\Set$). An {\em
element} of $F$ is an element of $F(C)$ for some $C\in\mathbf{C}$. A
{\em subset} of $F$ is a subfunctor of $\Forget\circ F$, i.e., a rule
$X$ that assigns to each $C\in\mathbf{C}$ a subset $X(C)\subseteq
F(C)$ in such a manner that $F_{C,D}(\phi)$ maps $X(C)$ into $X(D)$
for all morphisms $\phi\colon C\to D$. A {\em topological space over
$\mathbf{C}$} is a pair $(F,\mathcal{T})$ where $F$ is a functor as
above and $\mathcal{T}$ is a collection of subsets of $F$ including
the subsets $\emptyset,F$ that is closed under taking arbitrary intersections and finite unions.
\end{de}

\begin{re}
We note that all definitions that can be stated in terms of elements and
(closed) subsets of a topological space carry over to topological spaces
over $\mathbf{C}$. We also note that a topological space $(F,\mathcal{T})$
gives rise to a functor from $\mathbf{C}$ to the category of topological
spaces, which sends $C$ to the set $F(C)$ with the collection $\{X(C) \mid
X \in \mathcal{T}\}$ of closed subsets.  Clearly, not every functor from
$\mathbf{C}$ to the category of topological spaces arises in this manner.
\end{re}

In what follows, we use the term \textit{``injections"}
to refer to $R$-algebra monomorphisms. 

\begin{de} \label{de:SubsetAM}
Define $\AA_M$ to be the rule assigning to each $D \in \DomR$ the set $D \otimes M$. 
A subset of $\AA_M$ is a rule $X$ that assigns to each $D \in \DomR$ a subset $X(D)$ of $D \otimes M$ in such a manner that $\iota\otimes\id_M$ maps $X(D)$ into $X(E)$ for all injections $\iota\colon D\to E$. 
For every subset $S \subseteq R[M]$, the rule $\cV(S)$ assigning 
\[
D\mapsto\cV(S)(D):=\{m \in D \otimes M \mid \forall f \in S:f_D(m)=0\}
\]
is a subset of $\AA_M$. We say that $X\subseteq\AA_M$ is {\em closed}
if $X=\cV(S)$ for some $S\subseteq R[M]$. This collection of closed
sets makes $\AA_M$ into a topological space over $\DomR$ in the sense
of Definition~\ref{de:TopSpace}. 
\end{de}

\begin{re}
If $D$ is an $R$-domain, then we can make $D \otimes M$ into an
topological space by defining the closed subsets to be
$\cV(S)(D)$ for $S\subseteq R[M]$; we will call this the Zariski
topology (over $R$) on $D \otimes M$.  To see that these sets are preserved
under finite unions, one uses $\cV(S)(D) \cup \cV(T)(D)= \cV(S \cdot
T)(D)$, which holds since $D$ is a domain. For any $R$-algebra homomorphism $D \to E$ between
$R$-domains (not necessarily injective), the induced map $D \otimes
M \to E \otimes M$ sends $\cV(S)(D)$ into $\cV(S)(E)$.  Furthermore,
if $D \to E$ is injective, then that induced map is continuous with
respect to the topologies on $D \otimes M$ and $E \otimes M$. So
$\AA_M$ induces a functor from $\DomR$ to $\Top$ and the $\cV(S)$ are
closed subfunctors. In this paper, however, we will not consider closed subsets of $D\otimes M$ on their own. 
\end{re}

\begin{re}
We think of $\AA_M$ as the ``affine space'' corresponding to $M$.
Note that in the definition of closed subsets of $\AA_M$ we require $S$ to
be independent of~$D$, i.e., not every rule assigning to $D\in\DomR$ a subset of the form $\cV(S)(D)$ is a closed subset of $\AA_M$. To see that this is desirable, consider $R=\ZZ$, $M=R$ and let $X_n$ be the rule such that $X_n(D)=\{0\}=\cV(\{x\})(D)$ when $0<\cha D\leq n$ and $X_n(D)=D=\cV(\emptyset)(D)$ otherwise. Then $X_1\supseteq X_2\supseteq X_3\supseteq\ldots$ is a descending chain of rules and $X_{p-1}(\FF_p)=\FF_p\neq \{0\}=X_p(\FF_p)$ for every prime number $p>0$.
\end{re}

\begin{de}[Base change]
If $B$ is an $R$-algebra, and $D$ is a $B$-domain, then $D \otimes
M \cong D \otimes_B (B \otimes M)$ also carries a Zariski topology
over $B$, coming from closed sets defined by subsets of $B[B \otimes
M]$. This refines the Zariski topology on $D \otimes M$ over~$R$. If $X$
is a closed subset of $\AA_M$, then we write $X_B$ for the closed subset
of~$\AA_{B \otimes M}$ that maps a $B$-domain $D$ to $X(D)$.
\end{de}

Let $X$ be a subset of $\AA_M$. Then we define the {\em ideal} of $X$ to be 
\[
\cI_X:=\{ f\in R[M]\mid\forall D\in\DomR\forall x\in X(D): f_D(x)=0\}.
\]
As $f_D$ maps elements into a domain, we see that $\cI_X$ is a radical ideal of $R[M]$. We define the {\em closure} of $X$ in $\AA_M$ to be the closed subset $\overline{X}:=\cV(\cI_X)$ of $\AA_M$.

Let $\phi\colon M\to N$ be a polynomial law between finitely generated $R$-modules. Then the maps $(\phi_D)_{D\in\DomR}$ define a continuous map $\AA_M\to\AA_N$, i.e., for every injection $\iota\colon D\to E$, the diagram 
\[ 
\xymatrix{ \AA_M(D) \ar[r]^{\phi_D} \ar[d]_{\iota
\otimes \id_M} & \AA_N(D)
\ar[d]^{\iota \otimes \id_N} \\
\AA_M(E) \ar[r]^{\phi_E} \ar[r] &\AA_N(E)}
\]
commutes, so $\phi(X)=(D\mapsto \phi_D(X(D)))$ is a subset of $\AA_N$ for each subset $X$ of $\AA_M$, and for every subset $S\subseteq R[N]$, the subset 
\[
\phi^{-1}(\cV(S))=(D\mapsto\phi_D^{-1}(\cV(S)(D)))_D
\]
of $\AA_M$ is closed (as $\phi_D^{-1}(\cV(S)(D))=\cV(\phi^*S)(D)$ holds). As usual, we have
\[
\phi(\overline{X})\subseteq\overline{\phi(X)}
\]
for all subsets $X$ of $\AA_M$.

When $M$ is free and finitely generated, we have the usual correspondence between closed subsets and radical ideals. 

\begin{prop}
Let $M$ be a finitely generated free $R$-module of rank $n$.
Then the rule sending an element $x\in D\otimes M$ of $\AA_M$ to
$\fq_x:=\{f\in R[M]\mid f_D(x)=0\}\in\AA^n_R:=\Spec(R[M])$ is
surjective and maps closed subsets of $\AA_M$ to closed subsets
of $\AA^n_R$. Moreover, that map from closed subsets
of~$\AA_M$ to closed subsets of $\AA^n_R$ is a bijection. In particular, we have $\cI_{\cV(S)}=\rad(S)$ for any subset $S\subseteq R[M]$. 
\end{prop}
\begin{proof}
Note that for every $R$-domain $D$ and element $x\in D\otimes
M$, the set $\fq_x\subseteq R[M]$ is a prime ideal. Let
$\fq\subseteq R[M]=R[x_1,\ldots,x_n]$ be a prime ideal. Then we
have $\fq=\fq_x$ for $x=(x_1+\fq,\ldots,x_n+\fq)\in
(R[M]/\fq)\otimes M$. Next, let $S\subseteq R[M]$ be a set. Then
we see that $\{\fq_x\mid x\in \cV(S)(D),
D\in\DomR\}=\{\fq\in\Spec(R[M])\mid \fq\supseteq S\}$. So closed
subsets of $\AA_M$ are mapped to closed subsets of $\AA^n_R$.
Clearly, every closed subset arises from a closed subset of $\AA_M$. To see that this map is injective, we note that
$$
\cI_{\cV(S)}=\bigcap_{\substack{x\in \cV(S)(D)\\D\in\DomR}}\fq_x=\bigcap_{\substack{\fq\in\Spec(R[M])\\ \fq\supseteq S}}\fq=\rad(S)\mbox{ and }\cV(S)=\cV(\rad(S)). 
$$
Hence $\cV(S)$ is uniquely determined by its associated subset of $\AA^n_R$.
\end{proof}

While we have defined closed subsets of $\AA_M$ by looking at all
$R$-domains $D$, it actually suffices to look at algebraic closures
$\overline{K_\fp}$ where $\fp \in \Spec(R)$. For $\fp \in \Spec(R)$, we write $K_\fp:=\Frac(R/\fp)$ for the fraction field of $R/\fp$. 

\begin{prop}\label{prop:ideal_of_closure}
Let $X$ be a subset of $\AA_M$. Then 
\[
\cI_X=\bigcap_{\fp\in\Spec(R)}\left\{f\in R[M]\,\middle|\, f_{\kpbar}\in\cI_{X(\kpbar)}\right\}.
\]
\end{prop}
\begin{proof}
Clearly, the inclusion $\subseteq$ holds. Let $f\in R[M]$ be such that $f_{\kpbar}\in\cI_{X(\kpbar)}$ for all $\fp\in\Spec(R)$. Let $D$ be an $R$-domain and let $\fp$ be the kernel of the homomorphism
$R \to D$. Then there exists a field $L$ containing $\Frac(D)$ and
$\overline{K_{\fp}}$. By the Nullstellensatz, the fact that $f_{\kpbar}\in\cI_{X(\kpbar)}$ implies that $f_L\in\cI_{X(L)}$. It follows that $f_D$ vanishes on~$X(D)$.
\end{proof}

\begin{cor}\label{cor:Kpbarpoints}
A closed subset $X$ of $\AA_M$ is uniquely determined by its values
$X(\overline{K_\fp})$ where~$\fp$ runs over $\Spec(R)$.
\end{cor}
\begin{proof}
This follows from the previous proposition since $X=\cV(\cI_X)$.
\end{proof}

The proof of Theorem~\ref{thm:Main} in Section~\ref{sec:Proof} follows
a divide-and-conquer strategy in which the following two lemmas
and their generalisations to closed subsets of polynomial functors
(Lemmas~\ref{lm:Branching}
and~\ref{lm:integralextensions_polyfunctor}),
play a crucial role.

\begin{lm} \label{lm:Branching1}
Let $R$ be a ring with Noetherian spectrum and $r$ an element of $R$. Let
$\fp_1,\ldots,\fp_k$ be the minimal primes of $R/(r)$. Then
two closed subsets $X,Y\subseteq\AA_M$ are equal if and only if
$X_{R[1/r]}=Y_{R[1/r]}$ and $X_{R/\fp_i}=Y_{R/\fp_i}$ for
all $i=1,\ldots,k$.
\end{lm}
\begin{proof}
Suppose that $X_{R[1/r]}=Y_{R[1/r]}$ and $X_{R/\fp_i}=Y_{R/\fp_i}$ for
all $i=1,\ldots,k$. Let~$K$ be an $R$-field and let $R\to K$ be the corresponding homomorphism. If the image of $r$ in $K$ is zero, then $R \to K$ factors via $R/\fp_i$ for some $i =1, \ldots, k$ and hence $K$ is a $(R/\fp_i)$-domain. In this case, we have $X(K) = X_{R/\fp_i}(K)= Y_{R/\fp_i}(K)=Y(K)$. If the image of $r$ in $K$ is nonzero, then $K$ naturally is an $R[1/r]$-field. In this case, we have $X(K) = X_{R[1/r]}(K)= Y_{R[1/r]}(K)=Y(K)$. So $X=Y$ by Corollary~\ref{cor:Kpbarpoints}.
\end{proof}

\begin{lm}\label{lm:integralextensions_module}
Let $R\subseteq R'$ be a finite extension of domains and let $X,Y\subseteq\AA_M$ be closed subsets. Then $X=Y$ if and only if $X_{R'}=Y_{R'}$.
\end{lm}
\begin{proof}
The extension $R\subseteq R'$ satisfies lying over, i.e., for every prime $\fp\in\Spec(R)$ there is a prime $\fq\in\Spec(R')$ with $\fp=\fq\cap R$. The lemma follows by Corollary~\ref{cor:Kpbarpoints}.
\end{proof}

\subsection{Noetherianity of $\AA_M$} \label{ssec:TopNoeth}
We now prove Proposition~\ref{prop:TopHilbert}. Thus let $R$
be a ring. 

\begin{lm}
If $\Spec(R)$ is Noetherian, then so is $\Spec(R[x])$. 
\end{lm}
\begin{proof}
This is an application of \cite[Theorem 1.1]{erman-sam-snowden4} with trivial group.
\end{proof}

\begin{lm} \label{lm:FreeCase}
Assume that $\Spec(R)$ is Noetherian and set $N:=R^n$. Then $\AA_{N}$
is Noetherian, i.e., any chain $X_1 \supseteq X_2 \supseteq \cdots $
of closed subsets of $\AA_{N}$ stabilises eventually.
\end{lm}
\begin{proof}
Consider the chain $\cI_{X_1}\subseteq\cI_{X_2}\subseteq\cdots$ of radical ideals in $R[N] \cong
R[x_1,\ldots,x_n]$. Since the latter ring has a topological spectrum, this chain stabilises. 
Since $X_i=\cV(\cI_{X_i})$, so does the chain $X_1 \subseteq X_2 \subseteq \cdots$.
\end{proof}

\begin{proof}[Proof of Proposition~\ref{prop:TopHilbert}.]
Let $R$ be a ring with Noetherian spectrum, let $M$ be a finitely
generated $R$-module, and let $X_1 \supseteq X_2 \supseteq \cdots$ be
a chain of closed subsets of $\AA_M$. Since $M$ is finitely generated,
there exists a surjective $R$-module homomorphism $\phi\colon N:=R^n \to M$
for some $n$. This defines a (linear) polynomial law $N \to M$ and so a
continuous map $\AA_N \to \AA_M$. Set $Y_i:= \phi^{-1}(X_i)$, which is
the closed subset of $\AA_N$ such that $Y_i(D) = (1 \otimes \phi)^{-1}
(X_i(D))$ for all $R$-domains $D$. By Lemma~\ref{lm:FreeCase}, the chain $Y_1 \supseteq Y_2
\supseteq \cdots$ stabilises, i.e., $Y_n=Y_{n+1}$ for all $n \gg 0$. So,
since $1 \otimes \phi\colon D \otimes N \to D \otimes M$ is surjective for
every $R$-domain $D$, we have $X_i(D)=(1 \otimes\phi)(Y_i(D))$ for every $i$ and~$D$, and therefore $X_n=X_{n+1}$ for all $n\gg 0$.
\end{proof}

\begin{re} \label{re:SpecRM}
If two ideals $I$ and $J$ in $R[M]$ define the same closed subset in
$\Spec(R[M])$, then they have the same radical and hence define the
same closed subset in $\AA_M$. But it could possibly happen that two ideals
that define the same closed subset in $\AA_M$ do {\em not} define the
same closed subset in $\Spec(R[M])$. In particular, the proof above does
not show that $\Spec(R[M])$ is a Noetherian topological space. Indeed,
we don't know whether this is the case.
\end{re}

\begin{que}
Suppose that $\Spec(R)$ is Noetherian and let $M$ be a finitely generated
$R$-module. Is $\Spec(R[M])$ Noetherian? Is the map from
radical ideals of $R[M]$ to closed subsets of $\AA_M$ a
bijection?
\end{que}

\subsection{Dimension}

\begin{prop} \label{prop:DimensionConstructible}
Let $R$ be a domain, let $M$ be a finitely generated $R$-module and
let $X$ be a closed subset of $\AA_M$. Then the function 
\begin{align*}
\Spec(R) &\to\ZZ_{\geq -1}\\
\fp &\mapsto \dim_{\overline{K_\fp}}( X(\overline{K_\fp}))
\end{align*}
is constant in some open dense subset $\Spec(R[1/r])$ of $\Spec(R)$.
\end{prop}

\begin{proof}
By Lemma~\ref{lm:GenericFreeness}, there exists a nonzero $r\in R$
such that $R[1/r] \otimes M$ is free. It suffices to prove the
statement for the domain $R[1/r]$, the $R[1/r]$-module $R[1/r]
\otimes M$ and the closed subset $X_{R[1/r]}$ of
$\AA_{R[1/r]\otimes M}$. So we may assume that $M$ is free, say
of rank $m$, and so $X$ is a closed subset of $\AA^m_R$; let $I \subseteq R[x_1,\ldots,x_m]$ be its vanishing ideal. Choose an arbitrary monomial order on monomials in $x_1,\ldots,x_m$. For each nonzero $r \in
R$, let $M_r$ be the set of leading monomials of {\em monic}
polynomials in $R[1/r] \otimes I$; this is an upper ideal in the monoid of
monomials. By Dickson's lemma, there exists an $r$ such that $M_r$
is inclusion-wise maximal. Choose monic polynomials $f_1,\ldots,f_k
\in R[1/r][x_1,\ldots,x_n]$ whose leading
monomials generate the upper ideal~$M_r$. Then $f_1,\ldots,f_k$
generate the ideal $R[1/r] \otimes I$---indeed, otherwise there
would be some element $f$ in the latter ideal whose leading monomial is
not divisible by any of the leading monomials of the $f_i$; and letting
$r'$ be the leading coefficient of $f$ we would find that $M_{rr'}$
strictly contains $M_r$, a contradiction. Moreover,
again by maximality of $M_r$, the $f_i$ 
satisfy Buchberger's criterion: every S-polynomial of them
reduces to zero modulo $f_1,\ldots,f_k$ when working over
$R[1/r][x_1,\ldots,x_m]$. Then for each $\fp \in
\Spec(R[1/r])$, the images of the $f_i$ generate the ideal $K_\fp
\otimes I=K_\fp \otimes_{R[1/r]} (R[1/r]
\otimes I)$; and still satisfy Buchberger's criterion. Hence
these images form a Gr\"obner basis, and since the dimension
of $X(\overline{K_\fp})$ can be read of from the set of
leading monomials, that dimension is constant for $\fp \in
\Spec(R[1/r])$. 
\end{proof}

\begin{prop} \label{prop:Vanishing}
Let $R$ be a domain, $M$ a finitely generated $R$-module, and $X$ a closed
subset of $\AA_M$. Then there exists a nonzero $r \in R$ such that the
following holds: for any $f \in R[M]$, if $f$ vanishes identically on
$X(\overline{K})$, then $f$ vanishes identically on $X(\overline{K_\fp})$
for all $\fp \in \Spec(R[1/r])$.
\end{prop}

\begin{proof}
As in the previous proof, it suffices to prove the statement in the case that 
$M$ is free of rank $m$. Let $I \subseteq R[x_1,\ldots,x_m]$
be the vanishing ideal of $X$. This time, for each nonzero $r\in R$,
let $M_r$ be the set of leading monomials of {\em monic} polynomials in
$R[1/r][x_1,\ldots,x_m]$ {\em some power of which} lies in
$R[1/r] \otimes I$. Choose $r$ such that $M_r$ is maximal,
and $f_1,\ldots,f_k \in R[1/r][x_1,\ldots,x_m]$ monic, whose
powers lie in $R[1/r] \otimes I$, and whose leading monomials
generate the upper ideal $M_r$. Then the images of $f_1,\ldots,f_k$
form a Gr\"obner basis of the radical ideal of $K \otimes I$. Now
assume that $f \in R[M]$ vanishes identically on $X(\overline{K})$,
and let $g$ be the image of $f$ in $R[1/r][x_1,\ldots,x_m]$. Then
by the Nullstellensatz, some power of $g$ reduces to zero modulo
$f_1,\ldots,f_k$. But then that reduction holds modulo $\fp$ for every
$\fp \in \Spec(R[1/r])$, so $g$ vanishes identically on $X(\overline{K_{\fp}})$
for all such $\fp$.
\end{proof}

\section{Polynomial functors and their properties}
\label{sec:Polynomial functors over a ring}

\subsection{Polynomial functors over a ring}
For reasons that will become clear later, we will only be interested in
polynomial functors from the category $\fgfModR$ of finitely generated
free $R$-modules into either $\ModR$ or $\fgModR$.

\begin{de}
A polynomial functor $P\colon\fgfModR\to\ModR$ consists 
of an object $P(U) \in \ModR$ for each object $U\in\fgfModR$ and a polynomial law
\[
P_{U,V}\colon\Hom(U,V) \to \Hom(P(U),P(V))
\]
for each $U,V \in \fgfModR$ such that the diagram
\begin{center}
\begin{tikzcd}
\Hom(V, W) \oplus \Hom (U, V) \arrow[rr, "-\circ-"] \arrow[d,"P_{V,W}\oplus P_{U,V}"] && \Hom(U, W) \arrow[d,"P_{U,W}"] \\
\Hom(P(V), P(W)) \oplus \Hom(P(U), P(V)) \arrow [rr, "-\circ-"] && \Hom(P(U), P(W))
\end{tikzcd}
\end{center}
commutes for every $U,V,W \in \fgfModR$. Here the bilinear horizontal polynomial laws are given as in Remark \ref{rem:composition}. 
Moreover, for every $U \in \fgfModR$, we require that $P_{U,U}(\id_U) = \id_{P(U)}$ and we require that $P$ has {\em finite degree}, i.e., there is a uniform bound $d \in \ZZ_{\geq 0}$ such
that for all $U,V$ the polynomial law $P_{U,V}$ has degree at most $d$.
\end{de}

Polynomial functors $\fgfModR \to \ModR$ form an Abelian category $\PFR$ in
which a morphism $\alpha\colon Q \to P$ is given by an $R$-linear map $\alpha_U\colon Q(U) \to P(U)$ for each $U \in \fgfModR$ such
that the diagram of polynomial laws 
\begin{center}
\begin{tikzcd}
\Hom(U, V)\arrow[rr, "Q_{U,V}"] \arrow[d,"P_{U,V}"] &&\Hom(Q(U),Q(V)) \arrow[d,"\alpha_V\circ-"] \\
\Hom(P(U), P(V))\arrow [rr, "-\circ\alpha_U"] && \Hom(Q(U), P(V))
\end{tikzcd}
\end{center}
commutes for all $U,V$. Note that post-composing with $\alpha_V$ and pre-composing with~$\alpha_U$
are $R$-linear maps and hence, indeed, (linear) polynomial laws.

For every $R$-algebra $A$ and $R$-modules $U,V,W$, let $-\circ_A-$ be the $A$-bilinear extension of the $R$-bilinear composition maps $-\circ-\colon \Hom(V,W)\times\Hom(U,V)\to\Hom(U,W)$. So $(-\circ_A-)_A$ is the polynomial law extending $-\circ-$. Then the
diagram above says that
\begin{equation} \label{eq:PUVA}
P_{U,V,A}(\phi) \circ_A(1\otimes \alpha_U) =(1 \otimes \alpha_V )\circ_A Q_{U,V,A}(\phi)
\end{equation}
for all $R$-algebras $A$ and $\phi\in A\otimes\Hom(U,V)$. Note that to check that the diagram commutes, it suffices to check that this equality holds for $A=R[x_1,\ldots,x_n]$ and $\phi=x_1\otimes \phi_1+\cdots+x_n\otimes \phi_n$ where $\phi_1,\ldots,\phi_n$ is a basis of $\Hom(U,V)$.

Recall that for all $R$-modules $U,V$, there is a natural $A$-linear map 
\[
A \otimes \Hom(U,V) \to \Hom_A(A \otimes U,A \otimes V).
\]
For $U,V \in\fgfModR$, this map is an isomorphism. Thus an element $\phi$ of $A
\otimes \Hom(U,V)$ can be thought of as an ``element of $\Hom(U,V)$ with
coordinates in~$A$''. Viewing $Q_{U,V,A}(\phi),P_{U,V,A}(\phi)$
as maps, \eqref{eq:PUVA} implies that the diagram 
\begin{center}
\begin{tikzcd}
A\otimes Q(U)\arrow[rr, "\alpha_{U,A}"] \arrow[d,"Q_{U,V,A}(\phi)"] &&A\otimes P(U) \arrow[d,"P_{U,V,A}(\phi)"] \\
A\otimes Q(V)\arrow [rr, "\alpha_{V,A}"] && A\otimes P(V)
\end{tikzcd}
\end{center}
commutes; here $\alpha_{U,A}$ is the $A$-linear extension of $\alpha_U$. When $A$ is a polynomial ring over $R$,  the map 
\[
A \otimes \Hom(Q(U),P(V)) \to \Hom_A(A \otimes Q(U),A \otimes P(V))
\]
is injective and so the reverse implication also holds. So the family
$(\alpha_U)_U$ is a morphism of polynomial functors if and only if
the last diagram above commutes for all $A,U,V\phi$. This is
closer to the definition of polynomial functors over infinite
fields, and generalises as follows. 

\begin{de} \label{de:PolyTrans}
Let $P,Q$ be polynomial functors. We define a {\em polynomial
transformation} $\alpha\colon Q\to P$ be a rule assigning to
every $U\in \fgfModR$ a polynomial law $\alpha_U\colon Q(U) \to
P(U)$ such that the last diagram above commutes for all $R$-algebras $A$ and $\phi\in A\otimes\Hom(U,V)$.
\end{de}

Just like polynomial laws generalise $R$-module homomorphisms,
and the latter are precisely the linear polynomial laws, polynomial
transformations generalise morphisms of polynomial laws, and the latter
are precisely the linear polynomial transformations.

\begin{re}
If $R$ is an infinite field, then a polynomial functor 
\[
P\colon \fgfModR \to\fgModR=\fgfModR
\]
is a the same thing as a functor from the category
of finite-dimensional $R$-vector spaces to itself such that for all
$U,V \in \fgfModR$ the map 
\[
P_{U,V}\colon\Hom(U,V) \to \Hom(P(U),P(V))
\]
is a polynomial map. This is the set-up in \cite{draisma}. 
If $R$ is a field but not necessarily infinite, then a polynomial functor
$\fgfModR \to \fgfModR$ is a strict polynomial functor in the sense
of Friedlander-Suslin \cite{friedlander-suslin}.
\end{re}

Many of our proofs will involve passing to the case of (infinite) fields
and invoking arguments from \cite{draisma}. This is facilitated by
the following construction.

\begin{de}[Base change]
Let $B$ be an $R$-algebra and let $P\colon\fgfModR\to\ModR$ be a polynomial functor. Then $P$ induces a polynomial functor $P_B$ from $\fgfModB$ to $\ModB$ as follows: first, for each finitely generated free
$B$-module $U$ fix a $B$-module isomorphism $\psi_U\colon U \to B \otimes U_R$,
where $U_R$ is a free $R$-module of the same $R$-rank as the $B$-rank
of $U$.  Then, set $P_B(U):=B \otimes P(U_R)$. Next, for each
$B$-algebra $A$, we need to assign to every $\phi \in A \otimes_B \Hom_B(U,V)$ an image in $A \otimes \Hom_B(P_B(U),P_B(V))$. For this, note that
\[\begin{array}{rll}
A \otimes_B \Hom_B(U,V) &\cong A \otimes_B \Hom_B(B \otimes
U_R,B \otimes V_R) &\\
&\cong A \otimes_B (B \otimes
\Hom(U_R,V_R))\\ &\cong A \otimes \Hom(U_R,V_R), 
\end{array}\]
where the isomorphism in the first step is $1_A \otimes_B (\psi_V \circ-\circ \psi_U^{-1})$ and the second isomorphism follows from the freeness of $U_R$ and $V_R$. Via these isomorphisms, $\phi$ is mapped
to an element of $A \otimes \Hom(U_R,V_R)$. Applying $P_{U_R,V_R,A}$
to this element yields an element of $A \otimes \Hom(P(U_R),P(V_R)) \cong
A \otimes_B (B \otimes \Hom(P(U_R),P(V_R)))$, and applying the natural map
$B \otimes \Hom(P(U_R),P(V_R)) \to \Hom_B(B \otimes P(U_R),B \otimes
P(V_R))$ in the second factor (which may not be an isomorphism since
$P(U_R),P(V_R)$ need not be free) yields an element of $A \otimes_B
\Hom_B(P_B(U),P_B(V))$. It is straightforward to check that
$P_B$ thus defined is a polynomial functor from $\fgfModB$
to $\ModB$. A different choice of isomorphisms $\psi_U$
yields a different but isomorphic polynomial functor $P_B$.
\end{de}

\begin{re}
In this construction we have made use of the fact that $P$ is
a polynomial functor from finitely generated {\em free} $R$-modules
to $R$-modules. The choice of $\psi_U$'s could have been avoided as
follows: instead of working with $\fgfModR$, we could have worked with
the category whose objects are finite sets and whose morphisms $J \to I$
are given by $I \times J$ matrices with entries in $R$. Then $P_{J,I}$
would have been a polynomial law from the module of $I \times J$ matrices
to $\Hom(P(J),P(I))$. However, the set-up we chose stresses better that
we are interested in phenomena that do {\em not} depend on a choice of
basis in our free modules.
\end{re}

\begin{de}
A polynomial functor $P\colon\fgfModR\to\ModR$ is called
homogeneous of degree $d$ if the polynomial law $P_{U,V}$ is homogeneous of
degree $d$ for each $U,V\in\fgfModR$. 
\end{de}

Every polynomial functor $P\colon\fgfModR \to \ModR$ is a direct sum $P_0
\oplus \cdots \oplus P_d$, where $P_i\colon\fgfModR \to \ModR$
is the homogeneous polynomial functor of degree~$i$ given on
objects by $P_i(V)=\{v\in P(V)\mid
P_{V,V,R[t]}(t\otimes\id_V)(v)=t^i\otimes v\}$; and $P_{i,U,V}$
is the restriction of the degree-$i$ component of the polynomial
law $P_{U,V}$ to
$P_i(U)$. Here we identify $R[t]\otimes\Hom(P(V),P(V))$ with $\Hom(P(V),R[t]\otimes P(V))$.

\subsection{Duality} \label{ssec:Duality}

\begin{de} \label{de:Duality}
Let $P\colon\fgfModR \to \ModR$ be a polynomial functor over $R$.  Then we
obtain another polynomial functor $P^*\colon\fgfModR \to \ModR$ by setting,
for each $V \in \fgfModR$, $P^*(V):=P(V^*)^*=\Hom(P(V^*),R)$ and for
each $\phi \in A \otimes \Hom(U,V)$, 
\[ P^*_{U,V,A}(\phi):=P_{V^*,U^*,A}(\phi^*)^*, \]
where $\phi^*$ is the image of $\phi$ under the natural
isomorphism 
\[
A \otimes \Hom(U,V) \cong A \otimes \Hom(V^*,U^*)
\]
(here we use that $U,V$ are free) and the outermost $*$ again represents
a dual. 
\end{de}

The dual functor $P^*$ of $P$ has the same degree as $P$ and will play
a role in \S\ref{ssec:Condition2}. To avoid having too many stars,
we will there think of it as the functor that sends $V^*$ to $P(V)^*$.
If $P$ takes values in $\fgfModR$, then $(P^*)^*$ is canonically
isomorphic to $P$.   

\subsection{Shifting}
\label{ssec:Shifting}

Let $U$ be a finitely generated free $R$-module.

\begin{de}
We define the shift functor $\Sh_U\colon\fgfModR\to\fgfModR$ that sends $V\mapsto U \oplus V$ and $\phi\mapsto\id_U\oplus\,\phi$. For a polynomial functor $P\colon \fgfModR \to \fgModR$ we set $\Sh_U(P):=P \circ \Sh_U$, called the {\em shift of $P$ by $U$}.
\end{de}

\begin{lm} \label{lm:Shift}
The composition $\Sh_U(P)$ is again a polynomial functor $\fgfModR
\to \fgModR$, the projection $U \oplus V \to V$ yields a surjection of
polynomial functors $\Sh_U(P) \to P$ and inclusion the $V \to U \oplus
V$ yields a section $P \to \Sh_U(P)$ to that surjection. In particular,
$\Sh_U(P) \cong P \oplus (\Sh_U(P)/P)$. Furthermore, $\Sh_U(P)/P$ has degree
strictly smaller than the degree of $P$.
\end{lm}

\begin{proof}
The proof in \cite[Lemma 14]{draisma} (in the case where $R$ is an infinite
field) carries over to the current more general setting.
\end{proof}

\subsection{Dimension functions of polynomial functors} 

Let $P\colon\fgfModR \to \fgModR$ be a polynomial functor. For $\fp \in
\Spec(R)$, set $f_\fp(n):=\dim_{K_{\fp}}(K_\fp \otimes P(R^n))$. It
turns out that these functions are polynomials in $n$, and depend
semicontinuously on $\fp$. To formalise this semicontinuity, we order
polynomials in $\ZZ[x]$ by $f \geq g$ if $f(n) \geq g(n)$ for all $n
\gg 0$; this is the lexicographic order on coefficients.

\begin{prop} \label{prop:DimensionFunction}
For each $\fp \in \Spec(R)$ the function $f_\fp\colon\ZZ_{\geq 0} \to
\ZZ_{\geq 0}$ is a polynomial with integral coefficients of degree
at most the degree of $P$. Furthermore, the map $\fp \mapsto f_\fp$
is upper semicontinuous on $\Spec(R)$ in a strong sense: both the sets
$\{\fp \mid f_\fp \geq f \}$ and $\{\fp \mid f_\fp > f\}$ are closed
for all $f \in \ZZ[x]$.
\end{prop}
\begin{proof}
We proceed by induction on the degree of $P$. If $P$ has degree $0$,
then $P(R^n)$ is a fixed $R$-module $U$, and $f_\fp$ is the constant
polynomial that maps $n$ to $\dim_{K_\fp}(K_\fp \otimes U)$.  In this case,
if $f \in \ZZ[x]$ has positive degree, then $f_\fp > f$ and $f_\fp \geq
f$ are either both trivially true for all $\fp$ or both trivially false
for $\fp$ (depending on the sign of the leading coefficient of $f$),
so we need only look at constant $f$.

In this case, the result is classical; we recall the argument. Let $R^n
\to U$ be a surjective $R$-module homomorphism, and let $N$ be its
kernel. Since tensoring with $K_\fp$ is right-exact, $1 \otimes N$ spans
the kernel of the surjection $K_\fp^n \to K_\fp \otimes U$ for each $\fp$.

The statement that $\dim_{K_\fp}(K_\fp \otimes U)$ is upper semicontinuous
is therefore equivalent to the statement that dimension of the span of
$1 \otimes N$ in $K_\fp^n$ is lower semicontinuous.  And indeed, the
locus where this dimension is less than $k$ is defined by the vanishing
of all $k \times k$ subdeterminants of all $k \times n$ matrices (with
entries in $R$) whose rows are $k$ elements of $N$.

For the induction step, assume that the proposition is true for all
polynomial functors of degree $<d$ and assume that $P$ has degree $d \geq
1$. Then consider the functor $\Sh_R(P)$, which by Lemma~\ref{lm:Shift}
is isomorphic to $P \oplus Q$ for $Q:=\Sh_R(P)/P$ of degree $<d$.

By the induction hypothesis, the proposition holds for~$Q$: the function $g_\fp(n):=
\dim_{K_\fp}(K_\fp \otimes Q(R^n))$ equals a polynomial with integral
coefficients for all $n \geq 0$, and $\fp \mapsto g_\fp$ is
semicontinuous. Now we have
\[
\begin{array}{rll}
f_\fp(n+1)&= \dim_{K_\fp}(K_\fp \otimes P(R^1 \oplus R^n))&\\
&=\dim_{K_\fp} (K_\fp \otimes P(R^n)) + \dim_{K_\fp} (K_\fp \otimes Q(R^n))&=f_\fp(n) + g_\fp(n). 
\end{array} 
\]
This means that $f_\fp(n)$ is the unique polynomial
with $(\Delta f_\fp)(n):=f_\fp(n+1)-f_\fp(n)=g_\fp(n)$ for $n \geq 0$
and $f_{\fp}(0)=\dim_{K_\fp}(K_\fp \otimes P(0))$; this $f_\fp$ has integral coefficients and
degree at most $d$.

For the semi-continuity statement, note that $f_\fp \geq f$ is equivalent
to either $g_\fp=\Delta f_\fp > \Delta f$, or else $g_\fp \geq \Delta
f$ and moreover $f_\fp(0) \geq f(0)$. Both possibilities are closed
conditions on $\fp$. Similarly, $f_\fp > f$ is equivalent to either
$g_\fp > \Delta f$ or else $g_\fp \geq \Delta f$ and
$f_{\fp}(0)> f(0)$, which, again, are closed conditions.
\end{proof}

\subsection{Local freeness}

We now generalise Lemma~\ref{lm:GenericFreeness} to polynomial functors.

\begin{prop} \label{prop:GenericFreeness}
Let $R$ be a domain, $P\colon\fgfModR \to \fgModR$ a polynomial functor
and $S$ a subobject of $P$ in the larger category of polynomial
functors $\fgfModR \to \ModR$. Then there exists a nonzero $r \in R$
such that $R[1/r] \otimes S(U)$ and $R[1/r] \otimes P(U)$ are finitely
generated free $R[1/r]$-modules for all $U \in \fgfModR$, and the latter
is a direct sum of the former and another free $R[1/r]$-module.
\end{prop}

Note that we do not claim that the complement is itself the evaluation
of another subobject; i.e., $S_{R[1/r]}$ needs not be a summand of
$P_{R[1/r]}$ in the category of polynomial functors over $R[1/r]$.

\begin{proof}
Again, we proceed by induction on the degree of $P$. If $P$
has degree $0$, then so does $S$ and then the statement is just
Lemma~\ref{lm:GenericFreeness}. Suppose that the degree of $P$ is $d>0$
and that the proposition holds for all polynomial functors of degree
less than~$d$.

By Lemma~\ref{lm:Shift}, for each $n$ we have
\[ 
P(R^{n+1})=P(R^n) \oplus Q(R^n) 
\]
where $Q=\Sh_R(P)/ P$ has degree $<d$. 
Similarly, we have
\[
S(R^{n+1})=S(R^n) \oplus N(R^n) 
\]
where $N=\Sh_R(S) / S \subseteq Q$. It follows that 
\begin{align*} 
P(R^n)&=P(0) \oplus Q(0) \oplus Q(R^1) \oplus \cdots \oplus Q(R^{n-1}) \text{ and}\\
S(R^n)&=S(0) \oplus N(0)\oplus N(R^1) \oplus \cdots \oplus N(R^{n-1}).
\end{align*}
Now by Lemma~\ref{lm:GenericFreeness} there exists a nonzero $r_0$ such
that $R[1/r_0] \otimes P(0)$ is the direct sum of a free $R[1/r_0]$-module
and $R[1/r_0] \otimes S(0)$, which is also free. And by the induction
hypothesis there exists a nonzero $r_1  \in R$ such that for each $m$,
$R[1/r_1] \otimes Q(R^m)$ is a direct sum of two free
$R[1/r_1]$-modules, one of which is $R[1/r_1] \otimes N(R^m)$. Then $r:=r_0
\cdot r_1$ does the trick for the pair $P,S$.
\end{proof}

\subsection{The Friedlander-Suslin lemma}
\label{ssec:FriedlanderSuslin}

The Friedlander-Suslin lemma relates polynomial functors of bounded
degree to representations of certain associative algebras called {\em
Schur Algebras}. To introduce these, let $U \in \fgfModR$ and let $d\geq1$
be an integer. The bilinear polynomial law 
\[
- \circ -\colon  \End(U) \times \End(U) \to \End(U) 
\]
given by composition yields an algebra homomorphism  
\[ 
R[\End(U)] \to R[\End(U) \times \End(U)] \cong R[\End(U)] \otimes R[\End(U)] 
\]
which maps the part $R[\End(U)]_{\leq d}$ of degree $\leq d$ into 
\[
 \sum_{\substack{a,b \geq 0\\ a+b \leq d}} R[\End(U)]_a \otimes
R[\End(U)]_b \subseteq R[\End(U)]_{\leq d} \otimes
R[\End(U)]_{\leq d}. 
\]
Taking the dual $R$-modules, we obtain a map 
\[ 
R[\End(U)]_{\leq d}^* \otimes R[\End(U)]_{\leq d}^* \to 
(R[\End(U)]_{\leq d} \otimes R[\End(U)]_{\leq d})^* \to
R[\End(U)]_{\leq d}^*.
\]
We set $S_{\leq d}(U):=R[\End(U)]_{\leq d}^*$.  The first map is, in
fact, an isomorphism due to the fact that $S_{\leq d}(U)$ is finitely
generated and free as an
$R$-module. Indeed, if $U$ is free with basis $u_1,\ldots,u_n$, then
$\End(U)$ is free with basis $(E_{ij})_{i,j=1}^n$, where $E_{ij}u_k
=\delta_{jk} u_i$, and $R[\End(U)]_{\leq d}$ is free with basis the
monomials $x^{\alpha}$ of degree $\leq d$ in the coordinates $x_{ij}$
dual to the~$E_{ij}$, and hence $R[\End(U)]_{\leq d}^*$ is free with the
dual basis $(s_\alpha)_\alpha$, where $\alpha$ runs over all multi-indices
in $\ZZ_{\geq 0}^{n\times n}$ such that $|\alpha|:=\sum_{i,j} \alpha_{i,j}
\leq d$. We let $-*-\colon S_{\leq d}(U)\times S_{\leq d}(U) \to S_{\leq d}(U)$
be the bilinear map associated to the map above.

\begin{de}
The $R$-module $S_{\leq d}(U)$ with the bilinear map $-*-$ is called the {\em
Schur algebra of degree $\leq d$ on $U$}, and (given a basis of $U$),
the basis $(s_{\alpha})_\alpha$ is called its {\em distinguished basis}.
\end{de}

The Schur algebra is associative (but not commutative unless $n=1$); this
follows from the associativity of composition in $\End(U)$. Explicitly,
the coefficient of $s_{\gamma}$ in the product $s_\alpha * s_{\beta}$
is computed as follows: First, expand the composition $(\sum_{ij} x_{ij}
E_{ij}) \circ (\sum_{kl} y_{kl} E_{kl})$, where the $x_{ij}$ and $y_{kl}$
are variables, as $\sum_{i,l} (\sum_j x_{ij} y_{jl}) E_{il}=:\sum_{il}
z_{il} E_{il}$. Then expand $z^\gamma$ as a polynomial in the $x_{ij}$ and
the $y_{kl}$, and take the coefficient of the monomial $x^\alpha y^\beta$.

The map $\End(U) \to S_{\leq d}(U)$ that sends $\phi$ to the
$R$-linear evaluation map 
\[ 
R[\End(U)]_{\leq d} \to R,\quad f \mapsto f_R(\phi) 
\] 
is an injective homomorphism of associative $R$-algebras, so
$S_{\leq d}(U)$-modules $M$ are, in particular, representations of the
$R$-algebra $\End(U)$. In fact, they are precisely the {\em polynomial}
$\End(U)$-representations of degree $\leq d$, i.e., those
for which the map
$\End(U) \to \End(M)$ is not just a homomorphism of
(noncommutative) $R$-algebras but also
a polynomial law making certain diagrams commute. Since we will not
need this interpretation, we skip the details. 

Now suppose that $P$ is a polynomial functor $\fgfModR \to \ModR$
of degree $\leq d$. Then $P(U)$ naturally carries the structure of an
$S_{\leq d}(U)$-module as follows: the polynomial law
\[
P_{U,U}\colon\End(U) \to \End(P(U)) 
\]
has degree $\leq d$ and therefore we have
\[
P_{U,U,R[x_{11},x_{12},\ldots,x_{nn}]} \left(\sum_{i,j=1}^nx_{ij}\otimes E_{ij}\right) 
= \sum_{|\alpha| \leq d} x^\alpha\otimes \phi_\alpha 
\]
for certain endomorphisms $\phi_\alpha \in \End(P(U))$. Now the basis element $s_\alpha$
of $S_{\leq d}(U)$ acts on $P(U)$ via $\phi_\alpha$; it can
be shown that this construction is independent of the choice
of basis of $U$.

\begin{thm}[Friedlander-Suslin lemma, {\cite[Th\'eor\`eme 7.2]{touze}}
and {\cite[Theorem 3.2]{friedlander-suslin}}]
\label{thm:FriedlanderSuslin}
Let $U \in \fgfModR$ have rank $\geq d$. Then the association
$P \mapsto P(U)$ is an equivalence of Abelian categories from the full
subcategory of $\PFR$ consisting of polynomial functors $\fgfModR \to
\ModR$ of degree~$\leq d$ to the category of $S_{\leq d}(U)$-modules.
\end{thm}

To conclude this section, we observe that Schur algebras behave well
under base change: if $A$ is an $R$-algebra, then we have a
commuting diagram (up to natural isomorphisms):
\[ 
\xymatrix{
(\PFR)_{\leq d} \ar[r] \ar[d]_{P \mapsto P_A}& \{S_{\leq d}(U)\text{-modules}\}
\ar[d]^{M \mapsto A \otimes M}\\
(\PFA)_{\leq d} \ar[r] & \{(A \otimes S_{\leq d}(U))\text{-modules}\} 
}
\] 
where the lower horizontal map is evaluation at $A \otimes U$ and 
the $A$-algebra $A \otimes S_{\leq d}(U)$ is canonically
isomorphic to the Schur algebra $S_{\leq d}(A \otimes U)$ on the free
$A$-module~$A \otimes U$.

\subsection{Irreducibility in an open subset of $\Spec(R)$}
\label{ssec:IrredOpen}

Let $R$ be a domain and let $P\colon\fgfModR\to\fgModR$ be a polynomial functor. As before, for each prime $\fp\in\Spec(R)$ we set $K_\fp:=\Frac(R/\fp)$; in particular, $K:=K_{(0)}$ is the fraction field of $R$. Recall that the
base change functor yields a polynomial functor $P_{K_\fp}$ over the
field $K_\fp$ for each $\fp \in \Spec(R)$, and also a polynomial functor
$P_{\overline{K_{\fp}}}$ over the algebraic closure $\overline{K_{\fp}}$
of $K_{\fp}$. The goal of this section is to transfer certain properties
of $P_{K}$ to $P_{K_\fp}$ for $\fp$ in an open dense subset of
$\Spec(R)$.

\begin{prop} \label{prop:Irred}
Let $\overline{Q}$ be an irreducible subobject of $P_{K}$ in the
Abelian category of polynomial functors over $K$ and assume that
$\overline{Q}_{\overline{K}}$ is still irreducible. Then there exists
a subobject $Q$ of $P$ in the category of polynomial functors
$\fgfModR \to \ModR$ such that $Q_K=\overline{Q}$ and $Q_{\overline{K_\fp}}$ is an irreducible subobject
of $P_{\overline{K_\fp}}$ in the Abelian category of polynomial functors
over $\overline{K_\fp}$ for all primes $\fp$ in a dense open subset $\Spec(R[1/r]) \subseteq
\Spec(R)$.
\end{prop}

\begin{re}
Note that we don't require that $Q$ is a functor into $\fgModR$; we may
not be able to guarantee this if $R$ is not a Noetherian ring.
\end{re}

In order to prove this proposition, we use the following lemma.

\begin{lm} \label{lm:Irred2}
Let $A$ be a (not necessarily commutative) associative $R$-algebra and $N$ an $A$-module that is, as an $R$-module, finitely generated and free. Suppose that $\overline{K} \otimes N$ is an irreducible
$(\overline{K} \otimes A)$-module. Then there exists a dense open
subset $\Spec(R[1/r]) \subseteq \Spec(R)$ such that $\overline{K_\fp} \otimes N$ is an
irreducible $(\overline{K_\fp} \otimes A)$-module for all $\fp\in\Spec(R[1/r])$.
\end{lm}
\begin{proof}
Let $v_1,\ldots,v_n$ be an $R$-basis of $N$. For each $j \in [n]$ and each
$a \in A$ let $c_{a,i,j} \in R$ be the {\em structure constants} determined by
\[
a v_j = \sum_i c_{a,i,j} v_i. 
\]
For each $k=1,\ldots,n-1$, we will construct a constructible subset
$Z_k$ of the Grassmannian $\Gr_{R}(k,n)$ over $R$ whose set of
$\overline{K_\fp}$-points, for $\fp \in \Spec(R)$, is 
the set of $k$-dimensional $(\overline{K_\fp} \otimes A)$-submodules of $\overline{K_{\fp}} \otimes N$. The construction is as follows: for each $J \subseteq [n]$ of size $k$ consider the $k \times n$ matrix
$X_J$ whose entries on the columns labelled by $J$ are a $k \times k$ identity matrix over $R$ and whose other entries are variables
$x_{ij},i \in [k], j \in [n] \setminus J$. Recall that $\Gr_{R}(k,n)$
has an open cover of affine spaces $\AA_{R,J}^{k \times (n-k)}$ over $R$
on which the coordinates are precisely these $x_{ij}$ with
$j \not \in J$. For $j \in J$
we write $x_{ij} \in \{0,1\}$ for the corresponding entry of
$X_J$. Note that, for each $m=1,\ldots,k$ and each $a \in A$, we have
\[
(1 \otimes a)\left(\sum_{j=1}^n x_{mj} \otimes v_j\right)= \sum_{i=1}^n \sum_{j=1}^n c_{a,i,j} x_{mj} \otimes v_i\in R\left[x_{ij}\,\middle|\, i\in[k],j\in[n]\setminus J\right]\otimes N 
\]
and we define the row vector of coefficients
\[ y_{a,m}:=\left(\sum_{j=1}^n c_{a,i,j} x_{mj}\right)_{i=1}^n \]
with entries in the coordinate ring $R[x_{ij}\mid i\in[k],j\in[n]\setminus J]$ of $\AA_{R,J}^{k \times (n-k)}$. 

Let $C_J$ be the closed subset of $\AA_{R,J}^{k \times (n-k)}$ defined by the vanishing of all
$(k+1) \times (k+1)$-subdeterminants of the matrices
\[
\begin{bmatrix} y_{a,m} \\ X_J \end{bmatrix}
\]
for all choices of $a \in A$ and $m=1,\ldots,k$.  For each prime $\fp \in
\Spec(R)$, the subset $C_J(\overline{K_\fp}) \subseteq \Gr_R(k,n)(\overline{K_\fp})$
parameterises the $k$-dimensional $(\overline{K_{\fp}} \otimes A)$-submodules of $\overline{K_{\fp}} \otimes N\cong\overline{K_{\fp}}^{[n]}$ that map surjectively
to $\overline{K_\fp}^J$.  In particular, by the assumption that $\overline{K} \otimes
N$ is still irreducible, the image of $C_J$ in $\Spec(R)$ does not
contain the prime~$0$, for any $k$ and any $k$-set $J \subseteq [n]$. In
other words, the morphism $C_J \to \Spec(R)$ is not dominant. Set 
$Z_k:=\bigcup_{J \subseteq [n], |J|=k} \overline{C_J}$, a finite union
of locally closed subsets of the Grassmannian.  Then $Z_k \to \Spec(R)$
is still not dominant, and neither is $\left(\bigcup_{k=1}^{n-1}Z_k\right)
\to \Spec(R)$. Hence there exists a nonzero $r \in R$ that lies in the
vanishing ideal of the image; the open dense subset $\Spec(R[1/r]) \subseteq
\Spec R$ then has the desired property. 
\end{proof}

\begin{proof}[Proof of Proposition~\ref{prop:Irred}]
By the Friedlander-Suslin Lemma (Theorem~\ref{thm:FriedlanderSuslin})
and the fact that the Schur algebra behaves well under base change, it
suffices to prove the corresponding statement for all $d \in \ZZ_{\geq
0}$, $U:=R^d$, and all $S_{\leq d}(U)$-modules that are finitely generated
over $R$ (which, of course, is equivalent to being finitely generated
as an $S_{\leq d}(U)$-module).

So let $M$ be a finitely generated $S_{\leq d}(U)$-module and let
$\overline{N}$ be an irreducible $(K \otimes S_{\leq d}(U))$-submodule
of $K \otimes M$ that remains irreducible when tensoring with
$\overline{K}$. Define
\[ N:=\{v \in M \mid 1 \otimes v \in \overline{N} \}. \]
A straightforward computation shows that $N$ is a (not necessarily
finitely generated) $S_{\leq d}(U)$-submodule of $M$.

By Lemma~\ref{lm:GenericFreeness} there exist a nonzero $r \in R$ and
elements $v_1,\ldots,v_n \in N$ such that $R[1/r] \otimes N$ is a free
$R[1/r]$-module with basis $1 \otimes v_1,\ldots,1 \otimes
v_n$. Then Lemma~\ref{lm:Irred2} applied with
$R$ equal to $R[1/r]$ and $A$ equal to $R[1/r] \otimes
S_{\leq d}(U)$ shows that $\overline{K_{\fp}} \otimes N$ is 
an irreducible $(\overline{K_{\fp}} \otimes S_{\leq d}(U))$-submodule of
$\overline{K_{\fp}} \otimes M$ for all $\fp$ in some nonempty open
subset $\Spec R[1/(rs)] \subseteq \Spec(R[1/r]) \subseteq \Spec(R)$. 
\end{proof}

\subsection{Closed subsets of polynomial functors}
\label{ssec:ClosedSubsets}

Closed subsets of a polynomial functors play the role of affine varieties
in finite-dimensional algebraic geometry. In this subsection, $P$ is a
fixed polynomial functor $\fgfModR \to \fgModR$ of finite degree.

For any $U,V \in \fgfModR$ we have a sequence of polynomial laws
\[ 
\xymatrix{
\Hom(U,V) \times P(U) \ar[rr]^-{P_{U,V} \times\, \id} && \Hom(P(U),P(V))
\times P(U) \ar[rr]^-{(\phi,p) \mapsto \phi(p)}
&& P(V), 
}
\] 
whose composition we denote by $\Phi_{U,V}$. We also let $\Pi_{U,V}\colon \Hom(U,V) \times P(U)\to P(U)$ be the linear polynomial law given by projection. Recall that $\Phi_{U,V}$ and $\Pi_{U,V}$ both yield continuous maps from $\AA_{\Hom(U,V) \times P(U)}\to\AA_{P(V)}$.

\begin{de} \label{de:ClosedSubsetPF}
We define $\AA_P$ to be $P$. A {\em subset} of $\AA_P$ is a rule $X$ that assigns to each $U
\in \fgfModR$ a subset $X(U)$ of $\AA_{P(U)}$ (see 
Definition~\ref{de:SubsetAM}) in such a manner
that 
\[
\Phi_{U,V}(\Pi_{U,V}^{-1}(X(U))) \subseteq X(V)
\]
for all $U,V \in \fgfModR$. The subset $X\subseteq\AA_P$ is {\em closed} if $X(U)$ is a closed subset of~$\AA_{P(U)}$ for all $U
\in \fgfModR$. The {\em closure} of $X$ is the closed subset $\overline{X}$ of $\AA_P$ assigning $\overline{X(U)}$ to $U$ for all $U\in \fgfModR$.
\end{de}

It is worth spelling out what this means. Let $U,V$ be finitely generated free $R$-modules, let $D$ be an $R$-domain
and let $\phi \in D \otimes \Hom(U,V)$. Then the condition
is that $P_{U,V,D}(\phi)\in D\otimes \Hom(P(U),P(V))$ maps $X(U)(D)\subseteq D\otimes P(U)$ into $X(V)(D)$.
In the particular case where $V=U$, this condition can be informally
thought of as the condition that $X(U)$ is preserved under the polynomial
action of $\End(U)$. Let $\alpha\colon Q\to P$ be a polynomial transformation and let $X$ be a subset of $Q$. Then $\alpha(X)=(U\mapsto\alpha_U(X(U)))$ is a subset of $P$.

\begin{de}
For $X\subseteq\AA_P$, we define the ideal $\cI_X$ of $X$ to be the rule assigning $\cI_{X(U)} \subseteq R[P(U)]$ to $U$ for all $U\in \fgfModR$. The rule $\cI_X$ is an ideal in the $R$-algebra over the category $\fgfModR$ defined by $U \mapsto R[P(U)]$, i.e., for all $\phi\in\Hom(U,V)$ we have $\cI_X(V)\circ P_{U,V,R}(\phi)\subseteq \cI_X(U)$.
\end{de}

\begin{de}[Base change]\label{de: base change for X}
If $X\subseteq\AA_P$ is a closed subset and $B$ is an $R$-algebra, then we
obtain a closed subset $X_B$ of $\AA_{P_B}$ by letting, for a $U \in \fgfModB$,
$X_B(U)$ be the closed subset $X(U_R)_B$ of $\AA_{P_B(U)}=\AA_{B \otimes
P(U_R)}$, where $U_R$ is the free $R$-module such that $U
\cong B \otimes U_R$ from the definition of $P_B$.
\end{de}

We will use the following lemmas very frequently in our proof
of Theorem~\ref{thm:Main}.

\begin{lm} \label{lm:Branching}
Let $R$ be a ring with Noetherian spectrum and $r$ an element of $R$. Let
$\fp_1,\ldots,\fp_k$ be the minimal primes of $R/(r)$. Then
two closed subsets $X,Y\subseteq\AA_P$ are equal if and only if
$X_{R[1/r]}=Y_{R[1/r]}$ and $X_{R/\fp_i}=Y_{R/\fp_i}$ for
all $i=1,\ldots,k$.
\end{lm}
\begin{proof}
This follows from Lemma~\ref{lm:Branching1} with $X(U),Y(U)$ for every $U\in\fgfModR$.
\end{proof}

\begin{lm}\label{lm:integralextensions_polyfunctor}
Let $R\subseteq R'$ be a finite extension of domains and let $X,Y\subseteq\AA_P$ be closed subsets. Then $X=Y$ if and only if $X_{R'}=Y_{R'}$.
\end{lm}
\begin{proof}
This follows from Lemma~\ref{lm:integralextensions_module} with $X(U),Y(U)$ for every $U\in\fgfModR$.
\end{proof}

\begin{lm}\label{lm:biggestclosed1}
Let $U\in\fgfModR$ and $g\in R[P(U)]$. Then 
\[
Y(V)(D)=\{p\in D\otimes P(V)\mid \forall\phi\in D\otimes\Hom(V,U): g_D(P_{V,U,D}(\phi)(p))=0\}
\]
for all $V\in\fgfModR$ and $R$-domains $D$ defines a closed subset
$Y\subseteq\AA_P$. The subset $Y$ is the biggest closed subset of
$\AA_P$ such that $g$ is in the ideal of $Y(U)$. 
\end{lm}
\begin{proof}
It is easy to check that $Y(V)$ is a subset of $\AA_{P(V)}$ for all $V\in\fgfModR$ and that $Y$ is a subset of $\AA_P$. We need to check that $Y$ is a closed subset of $\AA_P$, i.e., that $Y(V)$ is a closed subset of~$\AA_{P(V)}$ for every $V\in\fgfModR$. 

Let $\phi_1,\ldots,\phi_n$ be a basis of $\Hom(V,U)$. For every $R$-algebra $A$, consider the map
\[
g_{A[x_1,\ldots,x_n]}\circ P_{V,U,A[x_1,\ldots,x_n]}(x_1\otimes \phi_1+\cdots+x_n\otimes\phi_n)\colon A[x_1,\ldots,x_n]\otimes P(V)\to A[x_1,\ldots,x_n].
\]
We have
\[
g_{A[x_1,\ldots,x_n]}\circ P_{V,U,A[x_1,\ldots,x_n]}(x_1\otimes \phi_1+\cdots+x_n\otimes\phi_n)|_{A\otimes P(V)}=\sum_{\alpha\in \ZZ_{\geq0}^n} x^{\alpha}g_{\alpha,A}
\]
where $g_{\alpha,A}\colon A\otimes P(V)\to A$. We get polynomial laws $g_{\alpha}=(g_{\alpha,A})_A\in R[P(V)]$. Set $S_V=\{g_\alpha\mid \alpha\in\ZZ_{\geq0}^n\}$. We claim that $Y(V)=\cV(S_V)$. Let $D$ be an $R$-domain and take $p\in Y(V)(D)$. Then, viewing $p$ as an element of $Y(V)(D[x_1,\ldots,x_n])$, we see that 
\[
g_{D[x_1,\ldots,x_n]}(P_{V,U,D[x_1,\ldots,x_n]}(\phi)(p))=0
\]
for all $\phi\in D[x_1,\ldots,x_n]\otimes\Hom(V,U)$. Using $\phi=x_1\otimes\phi_1+\cdots+ x_n\otimes \phi_n$, we get $p\in\cV(S_V)(D)$. Conversely, suppose that $p\in\cV(S_V)(D)$. Then
$$
g_{D[x_1,\ldots,x_n]}(P_{V,U,D[x_1,\ldots,x_n]}(x_1\otimes \phi_1+\cdots+x_n\otimes\phi_n)(p))=0
$$
Specializing the $x_i$ to elements of $D$, we find that 
$$
g_D(P_{V,U,D}(a_1\otimes \phi_1+\cdots+a_n\otimes\phi_n)(p))=0
$$
for all $a_1,\ldots,a_n\in D$. So $p\in Y(V)(D)$. So $Y(V)=\cV(S_V)$ is indeed closed.
\end{proof}

\begin{re}
It is not true in general that 
\[
Y(V)(D) = \{p \in X(V)(D) \mid\forall \phi \in \Hom(V, U): h_D(P(\phi)_D(p)) = 0 \}.
\]
For an example, take $R=\FF_p$, $P(V)=V$ and $h=x^p-x\in R[x]=R[P(R)]$. Then the right hand side above consists of all $p\in D\otimes V\cong D^n$ such that $x^p=x$ for every coordinate of $p$ while the left hand side also has the requirement that $(\alpha x)^p=\alpha x$ for all $\alpha\in E$ for every $D$-domain $E$. So $Y(V)(D)=0$.
\end{re}

\subsection{Gradings} \label{ssec:Gradings}

Let $P\colon \fgfModR \to \fgModR$ be a polynomial functor. For
each $U \in \fgfModR$, the $R$-algebra
$R[P(U)]$ has two natural gradings: first, the {\em ordinary}
grading that each coordinate ring $R[M]$ of a module $M$ has (see
Definition~\ref{de:RM}); and
second, a grading that takes into account the degrees of the homogeneous
components $P$, as follows. Write $P=P_0 \oplus P_1 \oplus \cdots
\oplus P_d$, so that $R[P(U)]$ is the tensor product of the $R[P_i(U)]$
by Proposition~\ref{prop:Products}. Then multiply the ordinary grading
on $R[P_i(U)]$ by $i$ and use these to define a grading on $R[P(U)]$,
called the {\em standard} grading. The standard grading has an alternative
characterisation, as follows: $f \in R[P(U)]$ is homogeneous of degree
$j$ if $f_A( P_{U,U,A}(a \otimes \id_U)(v))=a^j f_A(v)$ for all $A \in \AlgR$
and all $v \in A \otimes P(U)$. We have
\[
f_{A[t]}(v_0+tv_1+\cdots+t^dv_d)=\sum_{j=0}^{\infty}t^j f_{j,A}(v_0+v_1+\cdots+v_d)
\]
for all $A\in\AlgR$ and $v_i\in A\otimes P_i(U)$ where $f_j$ is the part of $f$ of standard degree $j$.

\begin{lm}
For any closed subset $X\subseteq\AA_P$ and any $U \in \fgfModR$, the ideal $\cI_X(U)$ is
homogeneous with respect to the standard grading.
\end{lm}
\begin{proof}
Take $f\in\cI_X(U)$ and let $D$ be an $R$-domain. Then 
\[
0=f_{D[t]}(P_{U,U,D[t]}(t\otimes \id_{U})(v_0+v_1+\cdots+v_d))=f_{D[t]}(v_0+tv_1+\cdots+t^dv_d)
\]
for all $v_i\in D\otimes P_i(U)$ such that $v_0+v_1+\cdots+v_d\in X(U)(D)$. Hence the homogeneous parts of $f$ are also contained in $\cI_X(U)$.
\end{proof}

\section{Proof of the main theorem} \label{sec:Proof}

\noindent In this section we prove Theorem~\ref{thm:Main}. Let $R$ be
a ring whose spectrum is Noetherian and let $P\colon\fgfModR \to \fgModR$ a polynomial functor of finite degree. We will prove that any chain $\AA_P
\supseteq X_1 \supseteq X_2 \supseteq\cdots$ of closed subsets eventually stabilises.

\subsection{Reduction to the case of a domain}
\label{ssec:ToDomain}

Since $\Spec(R)$ is Noetherian, the ring $R$ has finitely many minimal primes $\fp_1,\ldots,\fp_k$. By Lemma~\ref{lm:Branching} with $r=1$, the sequence $\AA_P\supseteq X_1 \supseteq X_2 \supseteq \cdots$ stabilises if and only if the sequence $\AA_{P_{R/\fp_i}}\supseteq X_{1,R/\fp_i} \supseteq X_{2,R/\fp_i} \supseteq \cdots$ stabilises for each $i\in[k]$. So from now on we assume that $R$ is a domain, we write $K_\fp:=\Frac(R/\fp)$ for $\fp \in \Spec(R)$, $K:=K_{(0)}=\Frac(R)$, and we let $\overline{K},\overline{K_\fp}$ be algebraic closures of $K,K_\fp$, respectively. 

\subsection{A stronger statement}

We will prove the following stronger statement which clearly implies Theorem~\ref{thm:Main}.

\begin{thm}\label{Sigma}
Let $(R,P,X)$ be a triple consisting of a domain $R$ with Noetherian spectrum, a polynomial functor $P\colon\fgfModR \to \fgModR$ of finite degree and a closed subset $X\subseteq \AA_P$. Then $(R,P,X)$ satisfies the following conditions:
\begin{enumerate}
\item Every descending chain $X=X_1 \supseteq X_2\supseteq \cdots$ of closed subsets of $X$ eventually stabilises.
\item There exists a nonzero $r \in R$ such that the following holds for all $U \in \fgfModR$: if $f \in R[P(U)]$ vanishes identically on $X(U)(\overline{K})$, then $f$ vanishes identically on $X(U)(\overline{K_\fp})$ for all primes $\fp \in \Spec(R[1/r])$.
\end{enumerate}
\end{thm}

\begin{re}
Condition (2) of the theorem means that $\cI_{X_{R[1/r]}}$ is determined by $\cI_{X_{\kbar}}$. More precisely, setting $R'=R[1/r]$, for every $U\in\fgfMod_{R'}$, the ideal 
\[
\cI_{X_{R'}}(U)=\cI_{X_{R'}(U)}\subseteq R'[P_{R'}(U)]
\]
is the pull-back of the ideal in $\kbar[P_{R'}(\kbar\otimes U)]$ of the affine variety $X_{R'}(\kbar\otimes U)$.
\end{re}

The proof of Theorem \ref{Sigma} is a somewhat intricate induction, combining
induction on $P$, Noetherian induction on $\Spec(R)$ and
induction on minimal degrees of functions in the ideal of $X$---for details,
see below.

\begin{no}
For any fixed triple $(R,P,X)$, we denote conditions $(1)$ and $(2)$ of Theorem~\ref{Sigma} by $\Sigma(R, P, X)$.
\end{no}

\subsection{The induction base}

If $P$ has degree zero, then $X$ is just a closed subset
of $\AA_{P(0)}$. Here, the Noetherianity statement is
Proposition~\ref{prop:TopHilbert} and the statement about vanishing
functions is Proposition~\ref{prop:Vanishing}.

\subsection{The outer induction} \label{ssec:OuterInduction}

To prove the theorem for $P$ of positive degree, we will show that $\Sigma(R,P,X)$ is implied by $\Sigma(R',P',X')$ where $X'$ is a closed subset of $\AA_{P'}$ and $(R',P')$ ranges over pairs that have one of the following forms:
\begin{enumerate}[label = (\roman*)]
\item $(R',P')=(R/\fp,P_{R/\fp})$ for some nonzero prime $\fp$ of $R$; or
\item $(R',P')$ where $R'$ is a domain that is a finite extension of a
localisation $R[1/r]$ of~$R$, $\deg P' \leq \deg P=:d$, for $K':=\Frac(R')$ we have $P'_{K'} \not \cong P_{K'}$ and for the largest~$e$ such that the homogeneous parts $P'_{e,K'}$
and $P_{e,K'}$ are not isomorphic, the former is a quotient of the
latter. 
\end{enumerate}
In both cases, we write $(R,P) \to (R',P')$. We consider the class $\Pi$ of all the pairs $(R, P)$. The reflexive
and transitive closure of the relation $\to$ is a partial order on  $\Pi$.

\begin{lm}\label{wellfoundedto}
The partial order on $\Pi$ is well-founded.
\end{lm}
\begin{proof}
Suppose that we had an infinite sequence 
\[
(R_0,P_0)\to(R_1,P_1) \to(R_2,P_2) \to \cdots
\]
of such steps. By the Friedlander-Suslin lemma, any sequence of steps
of type~(ii) only must terminate (see also \cite[Lemma 12]{draisma}). So
our sequence contains infinitely many steps of type (i).

Each step $(R,P) \to(R',P')$ induces a morphism $\alpha\colon \Spec(R') \to \Spec(R)$. This morphism $\alpha$ has the property that for irreducible closed subsets $C \subsetneq D\subseteq \Spec(R')$, we have $\overline{\alpha(C)}\subsetneq \overline{\alpha(D)}$. This holds trivially for steps of type (i), where the morphism $\alpha\colon\Spec(R/\fp)\to\Spec(R)$ is a closed embedding, and also for steps
of type (ii) by elementary properties of localisation and of integral extensions of rings (see, e.g., \cite[Corollary 4.18 (Incomparibility)]{eisenbud}).

Let $\alpha_i\colon\Spec(R_i)\to\Spec(R_{i-1})$ be the morphism
induced by $(R_{i-1},P_{i-1}) \to(R_i,P_i)$ and take
$\beta_i=\alpha_1 \circ \cdots \circ \alpha_i\colon \Spec(R_i) \to
\Spec(R_0)$. Then the maps $\beta_i$ have the same incomparability
property as the $\alpha_i$.  Hence, whenever the step $(R_{i-1},P_{i-1}) \to (R_{i},P_{i})$ is of type (i), there is the inclusion of irreducible closed sets $\im \alpha_i \subsetneq\Spec(R_{i-1})$ and therefore $\overline{\im \beta_i}\subsetneq\overline{\im \beta_{i-1}}$ is a strict inclusion. This contradicts the Noetherianity of~$\Spec(R_0)$.
\end{proof}

By Lemma~\ref{wellfoundedto} we can proceed by induction on $\Pi$,
namely, in proving that $\Sigma(R,P,X)$ holds, we may assume $\Sigma(R', P', X')$ whenever $(R',P')\leftarrow(R,P)$.

\begin{lm}\label{lm:sigmabranching}
Let $r\in R$ be a nonzero element and let $\fp_1,\ldots,\fp_k$ be the minimal primes of~$R/(r)$. Assume that $\Sigma(R[1/r],P_{R[1/r]},X_{R[1/r]})$ and $\Sigma(R/\fp_i,P_{R/\fp_i},X_{R/\fp_i})$ for each $i\in[k]$ hold. Then $\Sigma(R,P,X)$ holds as well.
\end{lm}
\begin{proof}
By Lemma \ref{lm:Branching}, we see that condition (1) for $(R,P,X)$ follows from condition (1) for $(R[1/r],P_{R[1/r]},X_{R[1/r]})$ together with $\Sigma(R/\fp_i,P_{R/\fp_i},X_{R/\fp_i})$ for each $i\in[k]$. Condition (2) for $(R,P,X)$ follows from condition (2) for $(R[1/r],P_{R[1/r]},X_{R[1/r]})$.
\end{proof}

Combining this lemma with our induction hypothesis, we see
that in order to prove $\Sigma(R,P,X)$ it suffices to prove
$\Sigma(R[1/r],P_{R[1/r]},X_{R[1/r]})$ for some $r\in R$. So we may
replace $(R,P,X)$ by $(R[1/r],P_{R[1/r]},X_{R[1/r]})$ whenever this
is convenient.

\subsection{Finding an irreducible factor}
\label{ssec:IrredFactor}
Now let $P\colon\fgfModR \to \fgModR$ be a fixed polynomial functor of degree $d>0$ over a domain $R$ with Noetherian spectrum. Recall that $K$ is the fraction field of $R$.

Suppose first that the base change $P_{K}$ has degree $<d$. Then $K \otimes P_d(U)=0$ for all $U \in \fgfModR$. In particular, this holds for $U=R^d$. So since $P_d(U)$ is a finitely generated $R$-module, there exists a nonzero $r \in R$ such that $R[1/r] \otimes P_d(U)=0$. By the Friedlander-Suslin lemma (Theorem~\ref{thm:FriedlanderSuslin}), we then find $(P_d)_{R[1/r]}=0$. In this case, we replace $(R,P,X)$ by $(R[1/r],P_{R[1/r]},X_{R[1/r]})$. By repeating this at most $d$ times, we may assume that the base change $P_{K}$ has the same degree as $P$.

We want a polynomial subfunctor $M$ of the top-degree part $P_d$ of
$P$ whose base change with $\overline{K}$ is an irreducible polynomial
subfunctor of $(P_d)_{\kbar}$. In the next lemma, we show that such an $M$ exists after passing from $R$ to a suitable finite extension of one of its localisations.

\begin{prop}\label{prop:IrreducibleFactor}
There exist a finite extension $R'$ of a localisation $R[1/r]$ of $R$ and a polynomial subfunctor $M$ of the top-degree part of the polynomial functor $P_{R'}$ such that the base change~$M_{\kbar}$ is an irreducible polynomial subfunctor of $P_{d,\kbar}$.
\end{prop}
\begin{proof}
The $S_d(\kbar^d)$-module $P_{d,\kbar}(\kbar^d) = \kbar \otimes P_d(R^d)$
is finite-dimensional and hence has an irreducible submodule $N'$. It
is finitely generated, say of dimension $n>0$. Let $\sum_j\alpha_{ij}
\otimes m_{ij}$ for $i = 1, \ldots, n$ be a $\kbar$-basis. By the
Friedlander-Suslin lemma, the irreducible submodule $N'$ corresponds
to an irreducible polynomial subfunctor~$N$ of~$P_{d,\kbar}$. The
elements $\alpha_i$ are algebraic over the fraction field $K$ of
$R$. Let $r\in R$ be the product of all the denominators appearing
in their minimal polynomials. Then $R' = R[1/r][\alpha_1, \ldots,
\alpha_n]$ is a finite extension of the localisation $R[1/r]$ of $R$
since the $\alpha_i$ are integral over $R[1/r]$. Consider the submodule
$M'$ of the $S_d(R'^d)$-module $P_{d,R'}((R')^d)$ generated by the elements
$\sum_j\alpha_{ij} \otimes m_{ij}$. By the Friedlander-Suslin lemma, $M'$
corresponds to a polynomial subfunctor $M$ of $P_{d,R'}$ whose base change
$M_{\kbar}=N$ is an irreducible polynomial subfunctor of $P_{d,\kbar}$.
\end{proof}

Let $r\in R$ and $R'$ be as in the previous proposition. We would like to reduce to the case where $R'=R$. As before, we can replace $(R,P,X)$ by $(R[1/r],P_{R[1/r]},X_{R[1/r]})$, so that $R'$ is a finite extension of $R$.  We now prove a version of Lemma~\ref{lm:sigmabranching} for such extensions.

\begin{lm}\label{lm:sigmabranching_integral}
Assume that $\Sigma(R',P_{R'},X_{R'})$ holds. Then $\Sigma(R,P,X)$ holds as well.
\end{lm}
\begin{proof}
By Lemma~\ref{lm:integralextensions_polyfunctor}, condition (1)
for $(R',P_{R'},X_{R'})$ implies condition (1) for $(R,P,X)$. Let
$r' \in R'$ be a nonzero element as in condition (2) for
$(R',P_{R'},X_{R'})$, i.e., for every $U\in\fgfMod_{R'}$, every
$f\in R'[P_{R'}(U)]$ vanishing identically on $X_{R'}(U)(\kbar)$
also vanishes identically on $X_{R'}(U)(\kpbar)$ for every prime ideal
$\fp\in\Spec(R'[1/r'])$. Now $(r')\cap R$ is not the zero ideal, since
$r'$ is nonzero and integral over $R$. Pick any nonzero $r \in (r')
\cap R$. We
claim that condition (2) holds for $(R,P,X)$ with this particular $r$.


Indeed, let $U_R\in\fgfModR$
and take $U:=R'\otimes U_R$. Let $f$ be an element of $R[P(U_R)]$
vanishing identically on $X(U_R)(\kbar)$. Then $f$ is naturally induces an element
of $R'[P_{R'}(U)]$ vanishing identically on $X_{R'}(U)(\kbar)=X(U_R)(\kbar)$. So
we see that $f$ vanishes on $X_{R'}(U)(\overline{K_{\fq}})$
for each $\fq\in\Spec(R'[1/r'])$. Since $R'$ is integral over $R$, for
any $\fp \in \Spec(R)$ there exists an $\fq
\in \Spec(R')$ with $\fq \cap R=\fp$; and if, moverover, the prime
ideal $\fp$ does not contain $r$, then the prime ideal $\fq$ does not
contain $r'$. Hence $f$ vanishes identically on $\overline{K_{\fp}}$, as
desired. 
\end{proof}

We replace $(R,P,X)$ by $(R',P_{R'},X_{R'})$, so that there exists a polynomial subfunctor~$M$ of the top-degree part $P_d$ of $P$ such that the base change $M_{\kbar}$ is an irreducible polynomial subfunctor of $P_{d,\kbar}$.

\subsection{Splitting off $M$} \label{ssec:SplittingOff}

Proposition \ref{prop:GenericFreeness} guarantees that after passing to a further localisation (and using Noetherian induction for the complement), we may assume that for each $U \in \fgfModR$, the $R$-module $P(U)$ is the direct sum of a finitely generated free $R$-module and the (also finitely generated free) $R$-module $M(U)$.
In particular, both $P$ and $P':=P/M$ are polynomial functors $\fgfModR \to \fgfModR$.

Let $\pi\colon P\to P'$ be the projection morphism. For a closed
subset $X\subseteq\AA_P$, we define the closed subset
$X'\subseteq\AA_{P'}$ as the closure of $\pi(X)$. Note that
$(R,P)\rightarrow(R,P')$ and hence $\Sigma(R,P',X')$ holds. In
particular, we may and will replace $R$ by a further localisation $R[1/r]$ which ensures that, if $f \in R[P'(U)]$ vanishes identically on $X'(U)(\kbar)$, then it vanishes identically on $X'(U)(\kpbar)$ for all $\fp \in \Spec(R)$.

\subsection{The inner induction} \label{ssec:InnerInduction}

We perform the same inner induction as in \cite[\S 2.9]{draisma}. Let
$\delta_X \in \{0,1,\ldots,\infty\}$ denote the smallest degree, in
the standard grading, of a homogeneous element of $R[P(U)] \cong
R[M(U)] \otimes R[P'(U)]$ (here we use that $P(U)$ is the direct sum
of the $R$-modules $M(U)$ and $P'(U)$), over all $U \in \fgfModR$,
that lies in the vanishing ideal of $X(U)$ but does not lie in the
vanishing ideal of the pre-image in $\AA_{P(U)}$ of $X'(U) \subseteq \AA_{P'(U)}$.
Note that $\delta_X=0$ is, in fact, impossible, since the coordinates
on $R[M(U)]$ have positive degree, so that a degree-$0$ homogeneous
element of $R[P(U)]$ that lies in the ideal of $X(U)$ is an element of
$R[P'(U)]$ that lies in the ideal of $X'(U)$. At the other extreme, $\delta_X=\infty$ means that $X(U)$ is the Cartesian product of $X'(U)$ with $\AA_{M(U)}$ for all $U$.
We order closed subsets of $\AA_P$ by $Y<X$ if either $Y' \subsetneq
X'$ or else $Y'=X'$ but $\delta_Y<\delta_X$. Note that, by the outer
induction hypothesis for $\Sigma(R, P', X')$ and since
$\{0,1,\ldots,\infty\}$ is well-ordered, this order is well-founded.
Hence when proving $\Sigma(P,R,X)$, we may assume that $\Sigma(P,R,Y)$ holds for all $Y<X$. 

First suppose that $\delta_X=\infty$. Then, for all proper closed
subsets $Y$ of $X$, we have $Y<X$ and so $\Sigma(R,P,Y)$ holds by the
inner induction hypothesis. It
follows that condition (1) holds for $(R, P, X)$. Condition (2) for
$(R,P,X)$ follows from condition (2) for $(R,P',X')$, with the same
$r \in R$ to be inverted. Indeed, if $f\in R[P(U)]\cong R[M(U)] \otimes R[P'(U)]$ vanishes
identically $X(U)(\kbar)\cong\AA_{M(U)}(\kbar)\times X'(U)(\kbar)$,
then, regarding $f$ as a polynomial in the coordinates on $M(U)$
with coefficients in $R[P'(U)]$, those coefficients must all vanish
identically on $X'(U)(\kbar)$, hence on $X'(U)(\overline{K_\fp})$ for
all $\fp \in \Spec(R[1/r])$.

\subsection{A directional derivative} \label{ssec:Directional}

Next, suppose that $1 \leq \delta_X<\infty$. Let $f \in R[P(U)]\cong
R[M(U)] \otimes R[P'(U)]$ be a homogeneous polynomial of degree
$\delta_X$ in the standard grading, which lies in the ideal of $X(U)$
but not on the preimage in $\AA_{P(U)}$ of $X'(U)$. Expanding $f$
as a polynomial in the coordinates on $R[M(U)]$ with coefficients in
$R[P'(U)]$, one of those coefficients does not lie in the ideal of
$X'(U)$. Our assumptions together with Corollary \ref{cor:Kpbarpoints}
guarantee that, in fact, that coefficient does not vanish identically on
$X'(U)(\overline{K})$, so that $f$ does not vanish identically on the
pre-image of $X'(U)(\overline{K})$ in~$\AA_{P(U)}(\overline{K})$. We
then proceed as in \cite[Lemma 18]{draisma}. Let $v_1, \ldots, v_m$
be an $R$-basis of $M(U)$ and extend this with $v_{m+1}, \ldots,
v_n$ to an $R$-basis of $P(U)$, inducing an isomorphism $R[P(U)]\cong
R[x_1,\ldots,x_n]$. The expression
\[
f_{R[x_1,\ldots,x_n,y_1,\ldots,y_m,t]}\left(\sum_{i=1}^n x_i \otimes v_i + \sum_{j=1}^m ty_j \otimes v_j\right)\in R[x_1,\ldots,x_n,y_1,\ldots,y_m,t]
\]
explicitly reads as
\[
 f(x_1+t y_1, x_2+t y_2,\ldots,x_m+t y_m, x_{m+1},\ldots,x_n).
\]
Take $p=1$ if $\cha R=0$ and $p=\cha R$ otherwise. A Taylor expansion
in $t$ turns this expression into 
\[
 f(x_1,\ldots,x_n) + t^{p^e}\cdot\left(h_1(x_1,\ldots,x_n) y_1^{p^e}+\cdots+h_m(x_1,\ldots,x_n) y_m^{p^e}\right) + t^{p^e+1}\cdot g
\]
for some integer $e\geq0$, polynomial $g\in R[x_1,\ldots,x_n,y_1,\ldots,y_m,t]$ and homogeneous polynomials $h_i \in R[P(U)]$ of (standard) degree $\delta_X-p^e d$ not all vanishing identically on $X(U)(\kbar)$. 
Specialising the variables $y_i$ to values $a_i\in\{0,1\}$, we get that
\[ h(x_1,\ldots,x_n):=\sum_{i=1}^m a_i^{p^e}h_i(x_1,\ldots,x_n)  \in R[P(U)] \]
does not vanish identically on $X(U)(\overline{K})$.

Let $p\in \kbar\otimes P(U)$ be a point in $X(U)(\kbar)$ such that
$h_{\kbar}(p)\neq 0$. Relative to the chosen basis of $P(U)$, we may
write $p = (\alpha_1,\ldots,\alpha_n)$. 
Reasoning as before, let $r \in R$ be the product
of all the denominators appearing in the minimal polynomials of the
$\alpha_i$ over $K$ so that $R' =  R[1/r][\alpha_1, \cdots, \alpha_k]$
is a finite extension of $R[1/r]$ containing all $\alpha_i$. Replacing
$R$ by $R'$ and using Lemma~\ref{lm:sigmabranching_integral}, we can therefore assume that $p\in X(U)(R)$ satisfies
$h_R(p)\neq0$. Further replacing $R$ by $R[1/h_R(p)]$, we find that
$h_D(p)\neq0$ for all $R$-domains $D$.  Define $Y$ to be the biggest
closed subset of~$X$ where~$h$ does vanish.

\begin{lm}\label{lm:biggestclosed}
We have 
\[
Y(V)(D)=\{p\in X(V)(D)\mid \forall\phi\in D\otimes\Hom(V,U): h_D(P_{V,U,D}(\phi)(p))=0\}
\]
for all $V\in\fgfModR$ and $R$-domains $D$.
\end{lm}
\begin{proof}
The closed subset $Y$ is the intersection of $X$ with the biggest closed subset of~$\AA_P$ where $h$ vanishes. So the lemma follows from Lemma~\ref{lm:biggestclosed1}.
\end{proof}

Let $X=X_1 \supseteq X_2 \supseteq \cdots$ be a sequence of closed
subsets of $X$. Since $Y<X$, the statement $\Sigma(R,P,Y)$ holds by the
inner induction. In particular, the intersections of the $X_i$ with $Y$
stabilise. This settles part of condition (1) of $\Sigma(R,P,X)$. We
now develop the theory to deal with the complement of $Y$. This will
afterwards be used to settle both condition (2) for $\Sigma(R,P,X)$
in \S\ref{ssec:Condition2} and complete the proof of condition (1)
in \S\ref{ssec:NoetherianityX}.

\subsection{Dealing with the localised shift}
\label{ssec:LocShift}

In \cite[Lemma 25]{draisma}, it is proved that for all $\fp \in \Spec(R)$ and $V \in \fgfModR$, the projection $\Sh_U(P)\to \Sh_U(P)/M$ induces a homeomorphism of $\Sh_U(X)[1/h](V)(\overline{K_\fp})$ with a closed subset of the basic open $(\Sh_U(P)/M)[1/h](V)(\overline{K_{\fp}})$. This proof uses that $M_{\overline{K}_\fp}$ is irreducible, which is why we have localised so as to make this true. The proof shows that, indeed, for each linear function $x \in (\overline{K_{\fp}} \otimes M(V))^*$, the $p^e$-th power $x^{p^e}$ lies in the sum of the ideal of $\Sh_U(X)[1/h](V)(\overline{K_\fp})$ in $\overline{K_\fp}[\kpbar \otimes P(U \oplus V)][1/h]$ and the subring $\overline{K_\fp}[\kpbar \otimes (P(U \oplus V)/M(V))]$.
We globalise this result as follows: for all $V \in \fgfModR$, define 
\[ 
N(V):=\left\{x \in M(V)^* \,\middle|\, x^{p^e} \in \cI_{\Sh_U(X)[1/h]}(V) + R[P(U \oplus V)/M(V)][1/h]\right\}. 
\]
There is a slight abuse of notation here: $M(V)$ is a submodule of $P(U \oplus V)$, so $M(V)^*$ is naturally a quotient of $P(U \oplus V)^*$ rather than a submodule.
But the projection $P(U \oplus V) \to P(U \oplus V)/M(V)$ admits a section (indeed, we have arranged things such that $P(U \oplus V)$ is isomorphic to the direct sum of the free $R$-modules $M(V)$ and $P(U
\oplus V)/M(V)$), and any section yields a section $M(V)^* \to P(U \oplus V)^*$.
Two such sections differ by adding elements from $(P(U \oplus V)/M(V))^*$, which is contained in the second term above, so $N(V)$ does not depend on the choice of section.

Recall from \S\ref{ssec:Duality} that $V^* \mapsto M(V)^*$ is a
polynomial functor $M^*$ of degree $d$. 

\begin{lm}
The association $V^* \mapsto N(V)$ is a polynomial subfunctor of $M^*$.
\end{lm}

\begin{proof}
Let $A$ be an $R$-algebra and take $V,W\in\fgfModR$. Take $y'\in A\otimes N(V)$ 
and $\phi^*\in A\otimes\Hom(V^*,W^*)$ corresponding to $\phi \in A \otimes \Hom(W,V)$. Then 
\begin{align*}
A\otimes\Hom(M(W),M(V)) &\cong A\otimes\Hom(M(V)^*,M(W)^*)\\ 
&\cong \Hom_A(A \otimes M(V)^*,A\otimes M(W)^*). 
\end{align*}
Denote the image of $M^*_{V^*,W^*,A}(\phi^*)=M_{W,V,A}(\phi)$ in $\Hom_A(A \otimes M(V)^*,A\otimes M(W)^*)$ by $M_{W,V,A}(\phi)^*$.
We need to show that $M_{W,V,A}(\phi)^*(y')\in A\otimes N(W)$.
This condition is $A$-linear in $y'$, so we may assume that
$y'=1 \otimes y$ with $y \in N(V)$. 

Choose $A=R[x_1,\ldots,x_n]$ and $\phi=\sum_i x_i\otimes\phi_i$
where the $\phi_i$ form a basis of $\Hom(W,V)$. Then in
particular we need that 
\[
M_{W,V,R[x_1,\ldots,x_n]}({\textstyle\sum_i} x_i\otimes\phi_i)^*(1\otimes y)\in R[x_1,\ldots,x_n]\otimes N(W).
\]
Conversely, by specializing the $x_i$ to $a_i\in A$ for any $R$-algebra $A$, this in fact suffices. As $M$ is a subfunctor of $P$, we may here replace $M$ by $P$.

Since $P(V)$ is free, the $R$-linear map 
\[
P_{W,V,R[x_1,\ldots,x_n]}({\textstyle\sum_i} x_i\otimes\phi_i)^*|_{P(V)^*}\colon P(V)^*\to R[x_1,\ldots,x_n]\otimes P(W)^*
\]
induces a homomorphism $\Phi\colon R[P(V)]\to R[x_1,\ldots,x_n]\otimes R[P(W)]$ of $R$-algebras. As taking the $p^e$-th power is additive, an element $z$ is contained in $R[x_1,\ldots,x_n]\otimes N(W)$ if and only if $z^{p^e}$ is contained in 
\[
R[x_1,\ldots,x_n]\otimes (\cI_{\Sh_U(X)[1/h]}(W) + R[P(U \oplus W)/M(W)][1/h]).
\]
So we now need to show that $\Phi(y)^{p^e}=\Phi(y^{p^e})$ is contained in this latter set. Since $y\in N(V)$, we have $y^{p^e}=g_1+g_2$ for some $g_1\in \cI_{\Sh_U(X)[1/h]}(V)$ and $g_2\in R[P(U \oplus V)/M(V)][1/h]$. Now we note that $\Phi(g_1)\in R[x_1,\ldots,x_n]\otimes \cI_{\Sh_U(X)[1/h]}(W)$ as in the proof of Lemma~\ref{lm:biggestclosed1} and $\Phi(g_2)\in R[x_1,\ldots,x_n]\otimes R[P(U \oplus W)/M(W)][1/h]$. So indeed 
\[
M_{W,V,R[x_1,\ldots,x_n]}({\textstyle\sum_i} x_i\otimes\phi_i)^*(1\otimes y)\in R[x_1,\ldots,x_n]\otimes N(W)
\]
holds.
\end{proof}

\begin{lm} \label{lm:NonzeroMult}
For every $V\in\fgfModR$, every element of $M(V)^*$ has a nonzero $R$-multiple
in~$N(V)$.
\end{lm}
\begin{proof}
By \cite[Lemma 25]{draisma}, any element $x$ of $M(V)^*$
has $1 \otimes x^{p^e} \in K \otimes N(V) \subseteq K \otimes M(V)^*$; in
the symbol $\subseteq$ we use that $M(V)$, and hence $M(V)^*$,
are free. Clearing denominators, we find that $r x^{p^e} \in
M(V)^*$ for some nonzero $r \in R$.
\end{proof}

\begin{lm}
There exists a nonzero $r \in R$ such that $R[1/r] \otimes N(V)
= R[1/r] \otimes M(V)^*$ holds for all $V \in \fgfModR$. 
\end{lm}
\begin{proof}
Recall that the degree of the polynomial functor $M$ is $d$ and consider
$V = R^d$. By Lemma~\ref{lm:NonzeroMult} and the fact that $M(V)$
is finitely generated, there exists a nonzero $r\in R$ such that
$R[1/r]\otimes N(V) = R[1/r] \otimes M(V)^*$.
The Friedlander-Suslin lemma, for polynomial functors over $R[1/r]$, gives
that then $R[1/r] \otimes N(V) = R[1/r] \otimes M(V)^*$ for every $V$.
\end{proof}

We now replace $R$ by the localisation $R[1/r]$ and may henceforth
assume that $N(V)=M(V)^*$.

\subsection{Proof of condition (2)} \label{ssec:Condition2}

To establish condition (2) for $(P,R,X)$, we will first prove an analogous
statement for the localised shift.

\begin{lm} \label{lm:Zprime}
There exists a nonzero $r \in R$ such that the following holds
for all $V\in \fgfModR$: if $g\in R[P(U\oplus V)]$ vanishes
identically on $\Sh_U(X)[1/h](V)(\overline{K})$, then $g$
vanishes identically on $\Sh_U(X)[1/h](V)(\overline{K_\fp})$ for all primes $\fp \in \Spec(R[1/r])$.
\end{lm}

\begin{proof}
Assume that $g\in R[P(U\oplus V)]$ vanishes identically
on $\Sh_U(X)[1/h](V)(\overline{K})$.
View $g$ as a polynomial in the
coordinates $x_i$ of $M(V)^*$ corresponding to a basis of $M(V)$
with coefficients in $R[P(U \oplus V)/M(V)]$.  By the conclusion of
\S\ref{ssec:LocShift}, we have $N(V) = M(V)^*$, which means that each
$x_i^{p_e}$ is a sum of an element in $R[P(U \oplus V)/M(V)][1/h]$
and an element in the ideal of $\Sh_U(X)[1/h](V)$.  We then find that also
$g^{p^e}=g_1 + g_2$ with $g_1 \in R[P(U \oplus
V)/M(V)][1/h]$ and $g_2 \in \cI_{\Sh_U(X)[1/h](V)}$.
Let $Z$ be the closure of the projection of $\Sh_U(X)[1/h]$ to $(\Sh_U(P)/M)[1/h]$.
Since both $g$ and $g_2$ vanish identically on
$\Sh_U(X)[1/h](V)(\overline{K})$, $g_1$ 
vanishes identically on $Z(V)(\overline{K})$. By the outer induction
hypothesis, after a localisation that doesn't depend on $g_1$ or on $V$,
one concludes that $g_1$ vanishes identically on $Z(V)(\overline{K_\fp})$
for all $\fp \in \Spec(R)$. But then $g^{p^e}$, and hence $g$ itself,
vanish identically on $\Sh_U(X)[1/h](V)(\overline{K_{\fp}})$.
\end{proof}

Now we can establish condition (2) of $\Sigma(R,P,X)$:

\begin{prop}\label{prop:condition2}
There exists a nonzero $r \in R$ such that the following holds for all $V\in \fgfModR$: if $g\in R[P(V)]$ vanishes identically on $X(V)(\overline{K})$, then $g$ vanishes identically on $X(V)(\overline{K_\fp})$ for all primes $\fp \in \Spec(R[1/r])$.
\end{prop}

\begin{re} \label{re:FixedRank}
For
each fixed $V$, such an $r$ exists by Proposition~\ref{prop:Vanishing}.
Taking the product of such $r$'s, the same
applies to a finite number of $V$'s, so we may restrict our attention
to all $V$ of sufficiently large rank; we will do this in the
proof.
\end{re}

\begin{proof}[Proof of Proposition~\ref{prop:condition2}]
By the inner induction hypothesis, after replacing $R$ by a localisation $R[1/r]$, we know that if $g \in R[P(V)]$ vanishes identically on $Y(V)(\overline{K})$, then it vanishes identically on $Y(V)(\overline{K_\fp})$ for all $\fp\in\Spec(R)$.

For any $V \in \fgfModR$ and $\fp\in\Spec(R)$, define $Z(V)(\kpbar):=X(V)(\kpbar) \setminus Y(V)(\kpbar)$. 
It suffices to show that with a further localisation we achieve that for any $V \in \fgfModR$, if $g \in R[P(V)]$ vanishes identically on all points of $Z(V)(\overline{K})$, then it vanishes identically on all points of $Z(V)(\overline{K_\fp})$ for all $\fp \in \Spec(R)$.
In proving this, by Remark~\ref{re:FixedRank} above, we may assume that $V$ has rank at least that of $U$. Hence we may replace $V$ by $U \oplus V$. 

Such a $g$ that vanishes identically on $Z(U \oplus V)(\kbar)$
vanishes, in particular, identically on
$\Sh_U(X)[1/h](V)(\overline{K})$. Lemma~\ref{lm:Zprime} says
that (after replacing $R$ by a localisation that does not depend
on $g$ or $V$), $g$ also vanishes identically on $\Sh_U(X)[1/h](V)(\overline{K_\fp})$ for all $\fp \in \Spec R$. This basic open is actually dense in $Z(U \oplus V)(\overline{K_\fp})$, as one sees as follows: $Z(U \oplus V)(\overline{K_\fp})$ is the image of the action
\[ 
\GL(\overline{K_{\fp}} \otimes (U\oplus V)) \times\Sh_U(X)[1/h](V)(\overline{K_\fp}) \to X(U \oplus V)(\overline{K_\fp}).
\] 
If the basic open were contained in the union of a proper subset of the
irreducible components of $Z(U\oplus V)(\overline{K_\fp})$, then, by irreducibility
of $\GL(\overline{K_{\fp}} \otimes(U\oplus V))$, so would the image of that action,
a contradiction. Hence $g$ then vanishes identically on
$Z(V)(\overline{K_\fp})$ for all $\fp \in \Spec(R)$. 
\end{proof}

\begin{re}
Note that, unlike $Y$, the $Z$ defined in the proof is not a
subset of $X$ in the sense of Definition~\ref{de:ClosedSubsetPF}. 
\end{re}

\subsection{Proof of the Noetherianity of $X$}
\label{ssec:NoetherianityX}

Finally, we prove condition (1) of $\Sigma(R,P,X)$. 
Let $X=X_1 \supseteq X_2 \supseteq \cdots$ be a sequence of
closed subsets of $X$. Recall from \S\ref{ssec:Directional} that the intersections of the $X_i$ with $Y$ stabilise. Now, consider again the projection $\Sh_U(P)[1/h] \to (\Sh_U(P)/M)[1/h]$. We let $Z_i'$ be the closure of the image of $\Sh_U(X_i)[1/h]$ in $(\Sh_U(P)/M)[1/h]$. Since the polynomial functor $(\Sh_U(P)/M)$ is smaller then $P$, we have Noetherianity for $(\Sh_U(P)/M)[1/h]$ and therefore the sequence $Z_1' \supseteq Z_2' \supseteq \cdots$ stabilises. We now conclude from this that the sequence of $\Sh_U(X_i)[1/h]$'s also stabilises.

\begin{lm}
Let $X''\subseteq X'\subseteq X$ be closed subsets, assume $\Sh_U(X'')[1/h]\subsetneq\Sh_U(X')[1/h]$ and let $Z''\subseteq Z'$ be the closures of their images in $(\Sh_U(P)/M)[1/h]$. Then $Z''\subsetneq Z'$. 
\end{lm}
\begin{proof}
Since $\Sh_U(X'')[1/h]\subsetneq\Sh_U(X')[1/h]$, we have 
\[
\Sh_U(X'')[1/h](V)\subsetneq\Sh_U(X')[1/h](V)
\]
for some $V\in\fgfMod_R$. This means that
$\cI_{\Sh_U(X'')[1/h]}(V)\supsetneq\cI_{\Sh_U(X')[1/h]}(V)$. Let
$g\in R[P(U\oplus V)][1/h]$ be an element of the former ideal
that is not contained in the latter. Then the same holds for
$g^{p^e}$. By the conclusion of \S\ref{ssec:LocShift}, $g^{p^e}$ is a sum of an element $g_1$ in $R[P(U \oplus V)/M(V)][1/h]$ and an element $g_2$ of $\cI_{\Sh_U(X)[1/h]}(V)\subseteq \cI_{\Sh_U(X')[1/h]}(V)$. This means that $g_1$ is also an element of $\cI_{\Sh_U(X'')[1/h]}(V)$ not contained in $\cI_{\Sh_U(X')[1/h]}(V)$. Hence
\[
\cI_{\Sh_U(X'')[1/h]}(V)\cap R[P(U \oplus V)/M(V)][1/h]\supsetneq\cI_{\Sh_U(X')[1/h]}(V)\cap R[P(U \oplus V)/M(V)][1/h]
\]
holds. The former ideal of $R[P(U \oplus V)/M(V)][1/h]$ equals $\cI_{Z''}(V)$ and the latter equals $\cI_{Z'}(V)$. So $Z''(V)\subsetneq Z'(V)$ and hence $Z''\subsetneq Z'$.
\end{proof}

By the lemma, the fact that the sequence of $Z_i'$ stabilises implies that the sequence of $\Sh_U(X_i)[1/h]$'s also stabilises. Now again, we write 
\[
Z_i(V)(\kpbar)=X_i(V)(\kpbar)\setminus Y(V)(\kpbar)
\]
for all $V\in\fgfModR$ and $\fp\in\Spec(R)$. We consider the descending sequence of $Z_i$'s. What is left to prove for the Noetherianity of $X$ is the following result.

\begin{lm}
The sequence $Z_1\supseteq Z_2\supseteq\cdots$ stabilises.
\end{lm}
\begin{proof}
Let $m$ be the rank of $U$. As in equation $(*)$ in \cite[\S 2.9]{draisma}, we have
\begin{eqnarray*}
Z_i(U\oplus V)(\kpbar) &=& \{p \in X_i(U \oplus V)(\kpbar) \mid h (g(p)) \neq 0 \text{ for some } g \in \GL(\kpbar \otimes(U \oplus V))\}\\
&=& \bigcup_{g \in \GL(\kpbar\otimes(U \oplus V))} g \Sh_U(X_i)[1/h](V)(\kpbar)
\end{eqnarray*}
for every $\fp \in \Spec(R)$. So the sequence of $Z_i$'s restricted to $V\in\fgfModR$ of rank~$\geq m$ stabilizes. As the sequence of $X_i(R^k)$'s stabilizes for each $k\in\{0,\ldots,m-1\}$ by Proposition~\ref{prop:TopHilbert}, the unrestricted sequence of $Z_i$'s also stabilizes.
\end{proof}

Since both the sequence of $X_i\cap Y$'s and $Z_i$'s stabilize,
using Corollary~\ref{cor:Kpbarpoints}, the sequence of $X_i$'s
also stabilizes. So the closed subset $X$ is Noetherian. This
concludes the proof of condition (1) for $(R,P,X)$ and hence the
proof of Theorem~\ref{thm:Main}. 

\subsection{Dimension functions of closed subsets of polynomial
functors} 

To illustrate that the proof method for Theorem~\ref{thm:Main}
can be used to obtain further results on closed subsets
of polynomial functors, we establish a natural common
variant of 
Propositions~\ref{prop:DimensionConstructible} and~\ref{prop:DimensionFunction}.
For each $\fp \in
\Spec(R)$ define the function $f_\fp:\ZZ_{\geq 0} \to \ZZ_{\geq 0}$
as $f_\fp(n):=\dim(X(R^n)(\overline{K_\fp}))$.

\begin{prop} \label{prop:Last}
For each $\fp \in \Spec(R)$, $f_{\fp}(n)$ is a polynomial in $n$ with
integral coefficients for all $n \gg 0$. Furthermore, the map that sends
$\fp$ to this polynomial is constructible.
\end{prop}

\begin{proof}[Proof sketch]
Both statements follow by inductions identical to the one for
Theorem~\ref{thm:Main}, using that, in the most interesting
induction step, for $n \geq m:=\rk(U)$ the dimension
of $X_{\overline{K_\fp}}(\overline{K_\fp}^n)$ is the
maximum of the dimensions of $Y_{\overline{K_\fp}}(\overline{K_\fp}^n)$
and \[(\Sh_U(X)[1/h])_{\overline{K}_\fp}(\overline{K_{\fp}}^{n-m}).\]
Furthermore, for the case where $X_{\overline{K_\fp}}$ is the pre-image
of $X_{\overline{K_\fp}}'$, we use Proposition~\ref{prop:DimensionFunction}, and for the base case
in the induction proof for the constructibility statement we use
Proposition~\ref{prop:DimensionConstructible}.
\end{proof}

\begin{ex}
Take $R=\ZZ$, take $P=S^3$, and let $X$ be the closed subset defined
as the image closure of the polynomial transformation $(S^1)^2 \to S^3, (v,w) \mapsto v^3 +
w^3$; see \S\ref{ssec:Applications} for similar polynomial
transformations. Then $X_{\overline{K_\fp}}(\overline{K_\fp}^n)$
has dimension $2n$ for $\fp \neq (3)$ and dimension $n$ for $\fp=(3)$,
since in the latter case the set of cubes of linear forms is a linear
subspace of the space of cubics. This is an instance of
Proposition~\ref{prop:Last}. 
\end{ex}

\subsection*{Acknowledgments}
AB, AD, and JD were partly, fully, and partly supported by Vici grant
639.033.514 from the Netherlands Organisation for Scientific Research
(NWO); JD was further partly supporte by SNSF project grant 200021\_191981. They thank Andrew Snowden and Daniel Erman for comments on an
earlier version of this paper, and Rob Eggermont for useful
discussions. 

\subsection*{Conflict of interest}
On behalf of all authors, the corresponding author states that there is no conflict of interest.

\end{document}